\documentclass[11pt]{amsart}
\usepackage{amsmath,mathrsfs,cite}

\usepackage{graphicx}
\usepackage[small]{caption}
\usepackage[colorlinks=true,urlcolor=blue,breaklinks=true,
citecolor=red,linkcolor=blue,linktocpage,pdfpagelabels,bookmarksnumbered,bookma rksopen]{hyperref}
\usepackage[english]{babel}
\usepackage[utf8]{inputenc}
\usepackage{enumerate}
\usepackage[hyperpageref]{backref}

\usepackage[left=3.1cm,right=3.1cm,top=3.15cm,bottom=3.15cm]{geometry}

\newcommand{\N}{{\mathbb N}}

\newcommand{\R}{{\mathbb R}}

\newcommand{\eps}{\varepsilon}
\newcommand {\Div} {\mbox{\rm div}} 

\numberwithin{equation}{section}
\newtheorem{theorem}{Theorem}[section]
\newtheorem{proposition}[theorem]{Proposition}
\newtheorem{lemma}[theorem]{Lemma}

\newtheorem{remark}[theorem]{Remark}

\newtheorem{definition}[theorem]{Definition}

\theoremstyle{definition}

\def\Xint#1{\mathchoice
	{\XXint\displaystyle\textstyle{#1}}%
	{\XXint\textstyle\scriptstyle{#1}}%
	{\XXint\scriptstyle\scriptscriptstyle{#1}}%
	{\XXint\scriptscriptstyle\scriptscriptstyle{#1}}%
	\!\int}
\def\XXint#1#2#3{{\setbox0=\hbox{$#1{#2#3}{\int}$ }
		\vcenter{\hbox{$#2#3$ }}\kern-.6\wd0}}

\def\dashint{\Xint-}

\newcommand{\bgset}[1]{\big\{#1\big\}}

\newcommand{\norm}[2][]{\left\|#2\right\|_{#1}}

\DeclareMathOperator{\dvg}{div}

\title[Eigenvalues for double phase variational integrals]{Eigenvalues for double phase variational integrals}

\author{Francesca Colasuonno}
\author{Marco Squassina}

\address[M. Squassina]{Dipartimento di Informatica \newline
Universit\`a degli Studi di Verona
\newline\indent
C\'a Vignal 2, Strada Le Grazie 15,
I-37134 Verona, Italy}
\email{marco.squassina@univr.it}

\address[F.\ Colasuonno]{Istituto per le Applicazioni del Calcolo ``M. Picone''\newline
Consiglio Nazionale delle Ricerche\newline\indent
Via dei Taurini 19,
00185 - Roma, Italy}
\email{f.colasuonno@iac.cnr.it}

\begin{document}
	



\begin{abstract}
We study an eigenvalue problem in the framework of double 
phase variational integrals and we introduce a sequence of nonlinear eigenvalues by a minimax procedure.
We establish a continuity result for the nonlinear eigenvalues with respect to the variations of the phases.
Furthermore, we investigate the growth rate of this sequence and get a Weyl-type law consistent with
the classical law for the $p$-Laplacian operator when the two phases agree.
\end{abstract}

\maketitle

\begin{center}
	\begin{minipage}{11cm}
		\small
		\tableofcontents
	\end{minipage}
\end{center}

\section{Introduction}

\subsection{Overview}
While the theory of linear eigenvalue problems is a well established topic of functional analysis \cite{courantH}, in the last few decades
many contributions were devoted to the study of nonlinear eigenvalue problems. As pointed out by
P.\ Lindqvist in his monograph \cite{Lindbook}, the work \cite{lieb} by E.H.\ Lieb was probably one of the first
containing  an interesting result about the minimum of a nonlinear Rayleigh quotient in several variables. Subsequently,
and especially in the first years of the nineties various papers were written by P.\ Lindqvist on the subject, we recall
here \cite{LindAMS,stab,stab2,Franzi,LL} and the comprehensive overview contained in his monograph \cite{Lindbook}.
\subsubsection{$p$-Laplacian}
More precisely, if $\Omega\subset\R^n$ is a smooth bounded domain and $p>1$, for the quasi-linear eigenvalue problem
\begin{equation}
\label{p-constant}
-{\rm div}(|\nabla u|^{p-2}\nabla u)=\lambda |u|^{p-2}u,\quad u\in W^{1,p}_0(\Omega),
\end{equation}
existence, regularity, qualitative properties and stability of eigenpairs $(u,\lambda)$ 
with respect to $p$ were investigated. The physical motivations that lead to the study of the eigenvalue
problem \eqref{p-constant} are mainly within the context of non-Newtonian fluids, dilatant for $p>2$ and pseudoplastic
for $1<p<2$, nonlinear elasticity and glaceology. Of particular relevance is the investigation of the properties of the first 
eigenpair $(u^1_p,\lambda^1_p)$, which corresponds to a solution
to the nonlinear minimization problem
$$
\lambda^1_p(\Omega):=\inf_{u\in \mathcal{M}_{p}} \|\nabla u\|_p^p,\qquad
\mathcal{M}_p:=\{u\in W^{1,p}_0(\Omega):\|u\|_{p}=1\}.
$$
In the one-dimensional case, $(u^1_p,\lambda^1_p)$ is explicitly determined by solving
the corresponding ODE boundary value problem. If $\Omega= (a, b)$, then 
$$
\lambda_p = (\pi_p/ (b-a))^{p-1},\qquad
\pi_p := 2(p-1)^{1/p}\int_0^1 (1-s^p)^{-1/p}ds,
$$
and
$$
u^1_p(x) =(p-1)^{-1/p} \sin_p (\pi_p(x-a)/(b-a)), 
$$
where $\sin_p$ is a $2\pi_p$-periodic function that generalizes the classical sine function \cite{1d}. Of course an analogous
analysis is not possible in the higher dimensional case.
The existence of a sequence of higher eigenvalues $(u^m_p,\lambda^m_p)$ can be obtained as a solution to 
\[
\lambda^{m}_p(\Omega):=\inf_{K\in\mathcal W^{m}_p}\max_{u\in K}\, \|\nabla u\|^p_{p},
\]
where 
\begin{equation}
\label{doppiwu}
\mathcal W^{m}_p:=\left\{K\subset\mathcal{M}_{p}
\,:\,  K \mbox{ symmetric $(K=-K)$ and compact},\, \gamma(K)\ge m\right\},
\end{equation}
and $\gamma(K)$ denotes the {\it Krasnosel'ski\u{\i} genus} of $K$. We recall that for every nonempty and symmetric subset $A$ of a Banach space $X$, its Krasnosel'ski\u{\i} genus is defined by
\begin{equation}
\label{krasno}
\gamma(A):=\inf\left\{k\in\mathbb{N}\, :\, \exists \mbox{ a continuous odd map } f:A\to\mathbb{S}^{k-1}\right\},
\end{equation}
where $\mathbb S^{k-1}$ is the unit sphere in $\mathbb R^k$, and with the convention that $\gamma(A):=+\infty$, if no such an integer $k$ exists. Actually one can define, in a similar fashion, a sequence
of higher variational eigenvalues by replacing the Krasnosel'ski\u{\i} genus with any other topological index $i$ which satisfies the
properties listed at the end of Section~\ref{sec4}. It is unknown whether these topological constructions 
exhaust the spectrum or not, which is the case for linear eigenvalue problems.
This is, in fact, one of the main open problems in the field since the appearance of these results. The sequence of eigenvalues 
$(\lambda_p^m)$ depend continuously, in smooth domains, on the value of $p$ (cf.\ \cite{stab,stab2,DM14,DM15,parini2011,Champ}) and, fixed the value
of $p$, they grow in $m$ according to a suitable Weyl-type law, $\lambda^m_p\approx C m^{p/n}$ for $m$ large, consistently with the celebrated
Weyl law for the linear case $p=2$, see e.g.\ \cite{MR1017063,peralGar}. 
For a complete investigation of the Stekloff spectrum for the pseudo $p$-Laplacian operator
$\sum_i D_i(|D_i u|^{p-2}D_i u)$, we refer the reader to \cite{bra-franz}.

\subsubsection{Fractional $p$-Laplacian}
We wish to point out that, very recently, 
a nonlocal version of the $p$-Laplacian, the fractional $p$-Laplacian $(-\Delta_p)^s$, 
\begin{equation}
\label{plap}
(- \Delta_p)^s\, u(x):= 2\, \lim_{\varepsilon \searrow 0} \int_{\mathbb{R}^n \setminus B_\varepsilon(x)}
\frac{|u(x) - u(y)|^{p-2}\, (u(x) - u(y))}{|x - y|^{n+s\,p}}\, dy, \qquad x \in \mathbb{R}^n.
\end{equation}
was introduced in \cite{LL}, where
properties of the first eigenvalue are investigated. Subsequently, Weyl-type laws were studied in \cite{meyanez} (an optimal
result, consistent with the local case, is not available yet, due to the strong nonlocal effects in the analysis)
and a complete analysis about the stability of variational eigenvalues was obtained in \cite{brasco}, with particular reference to the 
singular limit $s\nearrow 1$ towards the eigenvalues of the fractional Laplacian 
$(-\Delta)^s=\mathcal{F}^{-1}\,\circ \mathscr{M}_{s}\circ \mathcal{F}$,
where $\mathcal{F}$ is the Fourier transform operator and $\mathscr{M}_{s}$ is the multiplication by $|\xi|^{2\,s}$.
A rather complete analysis about the properties of the second eigenvalue was carried on in \cite{brapar}.

\subsubsection{$p(x)$-Laplacian}
Motivated by nonlinear elasticity theory and electrorheological fluids, problems involving 
variable exponents $p(x)$ were also investigated, especially in regularity theory 
(see e.g.\ \cite{acerbming,acerbming2,Base} and the references therein). Quite recently in \cite{Franzi},
nonlinear eigenvalues were investigated in this framework.
If $p:\overline{\Omega}\to\R^+$ is a log-H\"older continuous function and
 \begin{equation*}
 1<p_-:=\inf_{\Omega}p \leq p(x) \leq \sup_{\Omega}p =: p_+ < n\,\,\quad\mbox{for all $x\in\Omega$},
 \end{equation*} 
 the $m$-th (variational) eigenvalue $\lambda^{m}_{p(x)}$ can be obtained as 
 \begin{equation} 
 \label{px-seq}
 \lambda^{m}_{p(x)}:=\inf_{K\in\mathcal W^{m}_{p(x)}}\sup_{u\in K} \|\nabla u\|_{p(x)},
 \end{equation} 
where  $\|\cdot\|_{p(x)}$ is the Luxemburg norm  defined by
  \begin{equation*} 
  \label{Luxdef}
  \lVert u \rVert_{p(x)}:= 
  \inf \Big\{ \gamma > 0 : \int_{\Omega} \Big| \frac{u(x)}{\gamma} \Big|^{p(x)} dx\leq 1 \Big\}.
  \end{equation*}
and  $\mathcal W^{m}_{p(x)}$ is the set of symmetric, compact subsets of 
 $$
\mathcal{M}_{p(x)}:= \{u\in W^{1,p(x)}_0(\Omega):\|u\|_{p(x)}=1\}
 $$ 
 such that $i(K)\ge m$, where $i$ denotes the genus or any other topological index satisfying properties $(i_1)$--$(i_4)$ listed in Section \ref{sec4}. 
 In \cite{Franzi} existence and properties of the first eigenfunction were studied.
 The stability with respect to uniform perturbations of $p(x)$ was recently investigated in \cite{cs} (see also \cite{bm14}).
Finally,  the growth rate of the sequence in \eqref{px-seq} was investigated in \cite{persqu}, getting a 
 natural replacement for the constant case. 
 
 \subsubsection{Double phases}
 Given two constant exponents $q>p>1$, one can think about the case of a variable exponent $p(x)$ being a smooth approximation of 
 a discontinuous exponent $\bar p:\Omega\to(1,\infty)$ with $\bar p(x)=p$ if $x\in \Omega_1$ and
 $\bar p(x)=q$ if $x\in \Omega_2$, where  $\Omega=\Omega_1\cup \Omega_2$.
 In some sense, this situation can be interpreted as a {\em double phase} 
 behavior in two disjoint sub-domains of $\Omega$. A different kind of {\em double phase} situation occurs 
 for the energy functional
 \begin{equation} 
 \label{doublef-energy}
 u\mapsto \int_{\Omega} {\mathcal H}(x,|\nabla u(x)|)dx,\qquad
 {\mathcal H}(x,t):=t^p+a(x)t^q,\quad q>p>1,\,\,\,  a(\cdot)\geq 0,
 \end{equation}
 where the integrand switches two different elliptic behaviors. This defined in \eqref{doublef-energy} belongs
 to a family of functionals that Zhikov introduced to provide models of strongly
 anisotropic materials, see \cite{zhikov86,zhikov95,zhikov97} or \cite{zhikovBook} and the references therein.
Also,  \eqref{doublef-energy} settle in the context of the so called functionals with non-standard growth conditions,
according to a well-established terminology which was introduced by Marcellini \cite{Marcellini89,Marcellini91}, 
see also \cite{cup1,cup4,cup5,Marcellini96}.
In \cite{zhikovBook}, functionals \eqref{doublef-energy} are
used in the context of homogenization and elasticity  and the function $a$ drives the geometry of a composite of two
different materials with hardening powers $p$ and $q$.

\noindent
Significant progresses were recently achieved in the framework of regularity theory for 
minimisers of this class of integrands of the Calculus of Variations, see 
e.g.\ \cite{BarColMin1,BarColMin2,BarKuuMin,ColMing1,ColMing2}.

\subsection{Main results}
The main goal of this paper is to introduce a suitable notion of eigenpair associated with 
the energy functional \eqref{doublef-energy} consistently with the case $p=q$ and $a\equiv1$,
to prove the existence of an unbounded sequence of eigenvalues and, furthermore, get
continuity of each of these eigenvalues with respect to $(p,q)$ and a Weyl-type law
consistent with the classical (single phase) case.
Let $W^{1,\mathcal H}_0(\Omega)$ be the Musielak-Orlicz space introduced 
in Section~\ref{recalls} and
$$
\mathcal M_{{\mathcal H}}:=\{u\in W^{1,\mathcal H}_0(\Omega)\,:\, \|u\|_\mathcal H=1\},
$$ 
where $\Omega$ is a bounded domain of $\mathbb R^n$.
In the sequel we shall denote this set simply by $\mathcal M$.
\vskip2pt
\noindent
In the next theorems we assume that $1<p<q<n$ and that the following condition holds 
\begin{equation}
\label{cond-pq}
\dfrac{q}p<1+\dfrac{1}{n},\quad\,\,\,
\text{$\partial\Omega$ and $a:\overline{\Omega}\to [0,\infty)$ are Lipschitz continuous.}
\end{equation}

\noindent
The following are the main results of the paper.

\begin{theorem}[The first eigenpair] 
	\label{main1}
The first eigenvalue
	\begin{equation*}
	\lambda_\mathcal H^{1}:=\inf_{u\in {\mathcal M}} \|\nabla u\|_{\mathcal H}
	\end{equation*}
	 is positive and there exists a positive minimizer $u^1_\mathcal H\in \mathcal M\cap L^\infty(\Omega)$ which solves
	 \begin{equation}
	 \label{ELM}
	 \begin{aligned}-\mathrm{div}&\left(p\left|\frac{\nabla u}{\lambda}\right|^{p-2}\frac{\nabla u}{\lambda}+qa(x)\left|\frac{\nabla u}{\lambda}\right|^{q-2}\frac{\nabla u}{\lambda}\right)=\lambda S(u)(p|u|^{p-2}u+qa(x)|u|^{q-2}u),
	 \end{aligned}
	 \end{equation}
	 with $\lambda=\lambda_\mathcal H^{1}$, where
	 \begin{equation}\label{SuM}S(u):=\frac{\displaystyle\int_\Omega\left(p\left|\frac{\nabla u}{\lambda}\right|^p+qa(x)\left|\frac{\nabla u}{\lambda}\right|^q\right)dx}{\displaystyle\int_\Omega(p|u|^p+qa(x)|u|^q)dx}.
	 \end{equation}
The first eigenvalue $\lambda^1_\mathcal H$ is stable under monotonic perturbations of $\Omega$.\ If $a\equiv 1$, furthermore,  balls uniquely minimize the first eigenvalue among sets with a given $n$-dimensional Lebesgue measure. 
Finally, if $a\equiv1$ and, given a polarizer $H$ with $0\in H$ (resp.\ $0\in\partial H$), the domain $\Omega$ coincides with its polarization $\Omega^H$ (resp.\ its reflection $\Omega_H$), then there exists a first nonnegative eigenfunction having the symmetry $u=u^H$.
\end{theorem}
\smallskip

\noindent
The previous theorem is a consequence of a Poincar\'e-type inequality, which has been proved in the general framework of Musielak-Orlicz spaces in \cite{FanImbedding,HHK}.

\noindent
It is important to stress that, due to the presence of the term $S(u)$, equation~\eqref{ELM} turns out to be nonlocal. A similar nonlocal character arises in the case of the $p(x)$-Laplacian. 

\noindent
It would be interesting to investigate the {\em simplicity} of the first eigenvalue as well
as understanding if an arbitrary eigenfunction of {\em fixed sign} is automatically a first eigenfunction.
These issues are known to hold in the cases of the $p$-Laplacian and the fractional $p$-Laplacian
but, to the authors' knowledge, no result seems to be available in inhomogeneous settings like
the $p(x)$-Laplacian or the double phase operators. 
\vskip3pt
\noindent
In what follows, $i$ denotes the genus or any topological index satisfying $(i_1)$--$(i_4)$ in Section~\ref{sec4}. 

\begin{theorem}[Nonlinear spectrum]\label{nonlinspec} 
If $\mathcal W^{m}_\mathcal H$ is the set of symmetric compacts $K$ of 
	 ${\mathcal M}$ with index $i(K)\ge m$, then the sequence
	 $$
	 \lambda^{m}_\mathcal H:=\inf_{K\in\mathcal W^{m}_\mathcal H}\sup_{u\in K}\|\nabla u\|_\mathcal H,
	 $$
	 is non-decreasing, divergent, and for every $m\ge1$ there exists 
	 $u_m\in \mathcal M$ solving equation \eqref{ELM}, with 
	 $\lambda=\lambda_\mathcal H^{m}$.
\end{theorem}

\noindent
The result provides the construction of a variational spectrum for the double phase integrands
which is consistent with the single phase case $p=q$, see e.g. \cite{DM14,DM15}. The proof relies on the properties of the topological index $i$ and is based on the fact that the even functional $K|_\mathcal M: u\in\mathcal M\mapsto \|\nabla u\|_\mathcal H$ satisfies the (PS) condition. As for the $p$-Laplacian, it is not known whether the variational spectrum (i.e. the sequence of variational eigenvalues $(\lambda^m_\mathcal H)$) exhausts the whole spectrum (i.e. the set of all eigenvalues of \eqref{ELM}, see Definition \ref{eigen}) or not. In Theorem \ref{closed}, we prove that the spectrum is a closed set. While the following two results concern only the variational spectrum.  

\begin{theorem}[Stability]\label{main2} 
Let $(p_h,q_h)\searrow(p,q)$ as $h\to\infty$. Then
	$$
	\lim_{h\to\infty}\lambda^{m}_{\mathcal H_h}=\lambda^{m}_\mathcal H
	\quad\,\,\,\text{for every $m\geq 1$},
	$$
	where ${\mathcal H}_h(x,t):=t^{p_h}+a(x)t^{q_h}$ and ${\mathcal H}(x,t):=t^p+a(x)t^q$.
\end{theorem}

\noindent
The result implies, in particular, that each element of the sequence of nonlinear eigenvalues $(\lambda^{m}_{\mathcal H_h})$ for the double phase
case converges to the corresponding nonlinear eigenvalue for the $p_0$-Laplacian operator, whenever 
$(p_h,q_h)\searrow(p_0,p_0)$, as $h\to\infty$. The proof of Theorem \ref{main2} uses some recent results of \cite{DM14} and involves the $\Gamma$-convergence of a class of even functionals defined in $L^1(\Omega)$.

\vskip3pt
\noindent
In what follows, for any $E\subset \N$, we shall denote by $\sharp E$ the 
number of elements of $E$. Furthermore, $A\subset \mathbb R^n$ is called {\it quasi-convex} if there exists a constant $C>0$ such that for all $x,\,y\in A$ there is an arc joining $x$ to $y$ in $A$ having length at most $C|x-y|$.

\begin{theorem}[Weyl law]  
\label{weyl}
Let $\Omega$ be quasi-convex and let us set 
$$
w:=1+\|a\|_\infty+|\Omega|,\qquad  
\sigma:=n\left(\dfrac1p-\dfrac1q\right).
$$
Let $\lambda_\mathcal H^{m}$ be defined either through the genus $\gamma$ or through the $\mathbb Z_2$-cohomological index $g$.
Then, there exist $C_1,C_2>0$, depending only on $n$, $p$, $q$, such that 
$$
C_1|\Omega|(\lambda/w)^{n/(1+\sigma)}\le \sharp\big\{m\in\N: \lambda_\mathcal H^{m}
<\lambda\big\} \le C_2 |\Omega|(w\lambda)^{n/(1-\sigma)},
$$
for $\lambda>0$ large. In particular there exist 
$D_1,D_2>0$ depending on $n$, $p$, $q$, $a$ and $|\Omega|$ with 
$$
D_1 m^{(1-\sigma)/n}\leq \lambda_{\mathcal H}^m\leq D_2 m^{(1+\sigma)/n}
$$
for every $m\geq 1$ large enough.
\end{theorem}

\noindent
The result provides a consistent extension of the Weyl-type law for the $p$-Laplacian. 
We point out that, in the limiting case when $p=q$ and $a\equiv 1$, the Euler-Lagrange equation~\eqref{ELM} (see also the formulation \eqref{EL}) reduces to the usual quasi-linear problem \eqref{p-constant} involving $-\Delta_p$, but the eigenvalues $\lambda^m_{\mathcal H}$ and $\lambda_p^m$, in light of their definition, do satisfy
$$
\lambda^m_{\mathcal H}=(\lambda_p^m)^{\frac{1}{p}}\,\,
\quad\text{for all $m\in\N$.}
$$
Hence, the estimate for the growth of $(\lambda_{\mathcal H}^m)$ of Theorem~\ref{weyl} for the case $p=q$ 
(formally corresponding to $\sigma=0$)
 $$
 D_1 m^{1/n}\leq \lambda_{\mathcal H}^m\leq D_2 m^{1/n}
 $$
 should be compared with the estimate $D_1 m^{p/n}\leq \lambda_p^m\leq D_2 m^{p/n}$ and, thus, it is
consistent with the results obtained in \cite{MR1017063,peralGar}. For the linear case
$q=p=2$ and $a=1$, we mention the pioneering contribution by H.\ Weyl \cite{weyl},
from which these type of estimates inherit the name.
The proof of Theorem \ref{weyl} relies on the properties and on the relations among three different topological indices (i.e. the genus, the cogenus, and the cohomological index) and is an adaptation to the double phase setting of an idea developed in \cite{persqu} for the $p(x)$-Laplacian operator. 

\subsection{Plan of the paper}
In Section \ref{recalls} we give some basic definitions and useful results on Musielak-Orlicz spaces and in particular on the spaces generated by the $N$-function $\mathcal H$ as in \eqref{doublef-energy}.
In Section \ref{sec3} we derive the Euler-Lagrange equation corresponding to the minimization of the Rayleigh ratio $\|\nabla u\|_\mathcal H/\|u\|_\mathcal H$ and we prove Theorem \ref{main1} concerning the first eigenpair $(\lambda^1_\mathcal H,u^1_\mathcal H)$ and some useful properties of the spectrum, such as its closedness and the behavior of the first eigenvalue for large exponents $p$ and $q$. 
Section \ref{sec4} contains the definition of the variational eigenvalues of \eqref{ELM} and the proof of Theorem~\ref{nonlinspec}, while
Section \ref{sec5} is devoted to the proof, via $\Gamma$-convergence, of the stability of the nonlinear spectrum (Theorem \ref{main2}), i.e. the continuity of the eigenvalues with respect to the variation of the phases $p$ and $q$ from the right. Finally, in 
Section \ref{sec6} we study the asymptotic growth of the variational eigenvalues and prove Theorem~\ref{weyl}.

\section*{Acknowledgments} 
\noindent
The authors warmly thank Giuseppe Mingione for various suggestions about the problem and 
its related literature. Paolo Baroni is also acknowledged for some hints on the application of the  
Harnack inequality of \cite{BarColMin1} in the proof of Theorem~\ref{main1}.
The statement and the proof of Lemma \ref{brasco1} are attributed to Lorenzo Brasco. 
Part of this work was written during a visit of Francesca Colasuonno  
at the Department of Computer Science of the University of  Verona.\ The hosting institution is gratefully acknowledged.
The research  was partially supported by Gruppo Nazionale per l'Analisi Matematica,
la Probabilit\`a e le loro Applicazioni (INdAM).

\smallskip

\section{Preliminary results}\label{recalls}
\subsection{Musielak-Orlicz spaces} We recall here some notions on Musielak-Orlicz spaces, see for reference \cite{Musielak}, Section 2 of \cite{Base}, and also Section 1 of \cite{Fan12}.  
Let $\Omega\subset\mathbb R^n$ be a bounded domain.  
\begin{definition}\rm A continuous, convex function $\varphi:[0,\infty)\to [0,\infty)$ is called {\it $\Phi$-function} if $\varphi(0)=0$ 
and $\varphi(t)>0$ for all $t>0$.
\medskip

\noindent
A function $\varphi:\Omega\times[0,\infty)\to [0,\infty)$ is said to be a {\it generalized $\Phi$-function}, denoted by $\varphi\in\Phi(\Omega)$, if $\varphi(\cdot,t)$ is measurable for all $t\ge0$ and $\varphi(x,\cdot)$ is a $\Phi$-function for a.a. $x\in\Omega$. 
\medskip

\noindent
$\varphi\in\Phi(\Omega)$ is {\it locally integrable} if $\varphi(\cdot,t)\in L^1(\Omega)$ for all $t>0$.
\medskip 

\noindent
$\varphi\in\Phi(\Omega)$ satisfies the {\it $(\Delta_2)$-condition} if there exist a positive constant $C$ and a nonnegative function $h\in L^1(\Omega)$ such that 
$$
\varphi(x,2t)\le C\varphi(x,t)+h(x)\quad\mbox{for a.a. }x\in\Omega\mbox{ and all }t\in[0,\infty).
$$
Let $\varphi,\,\psi\in \Phi(\Omega)$. The function $\varphi$ {\it is weaker than} $\psi$, denoted by $\varphi\preceq \psi$, if there exist two positive constants $C_1,\, C_2$ and a nonnegative function $h\in L^1(\Omega)$ such that 
$$\varphi(x,t)\le C_1\psi(x,C_2 t)+h(x)\quad\mbox{for a.a. }x\in\Omega\mbox{ and all }t\in [0,\infty).$$
\end{definition}

\noindent
Given $\varphi\in \Phi(\Omega)$, the {\it Musielak-Orlicz space} $L^\varphi(\Omega)$ is given by
$$L^\varphi(\Omega):=\left\{u:\Omega\to\mathbb R\mbox{ measurable }:\,\exists \,\gamma>0\mbox{ s.t. }\varrho_\varphi(\gamma u)<\infty \right\},$$
where $$\varrho_\varphi(u):=\int_\Omega\varphi(x,|u|) dx$$ is the modular, while 
$$\|u\|_\varphi:=\inf\left\{\gamma>0\,:\,\varrho_\varphi(u/\gamma) \le 1\right\}$$
is the norm defined on $L^\varphi(\Omega)$.

\begin{proposition}\label{Banach}{\rm (cf.\ \cite[Theorem 7.7]{Musielak})} 
Let $\varphi\in \Phi(\Omega)$, then $(L^\varphi,\|\cdot\|_\varphi)$ is a Banach space.
\end{proposition}

\begin{proposition}\label{embeddingL}{\rm (cf.\ \cite[Theorem 8.5]{Musielak})} 
Let $\varphi,\,\psi\in \Phi(\Omega)$, with $\varphi\preceq\psi$. Then 
$$
L^\psi(\Omega)\hookrightarrow L^\varphi(\Omega).
$$
\end{proposition}

\begin{proposition}[unit ball property]\label{properties}
Let $\varphi\in \Phi(\Omega)$, then the following properties hold.
\begin{itemize}
\item[$(i)$] If $\varphi$ satisfies $(\Delta_2)$, then $L^\varphi(\Omega)=\{u:\Omega\to\mathbb R\mbox{ measurable }:\,\varrho_\varphi(u)<\infty \}$, {\rm(cf.\ Theorem 8.13 of \cite{Musielak})}.
\item[$(ii)$] If $u\in L^\varphi(\Omega)$, then  $\varrho_\varphi(u)<1$ (resp. $=1;\,>1$) $\Leftrightarrow$ $\|u\|_\varphi<1$ (resp.\ $=1;\,>1$), {\rm (cf.\ Lemma 2.1.14 of \cite{Base})}.
\end{itemize}
\end{proposition}

\begin{remark}
{\rm As a consequence of the homogeneity of the norm and of the unit ball property (Proposition \ref{properties}-$(ii)$) we have the following implications: 
\begin{equation}\label{norm-mod}\|u\|_{\varphi_1}=\|v\|_{\varphi_2}\quad\Rightarrow\quad\left\|\frac{u}{\|v\|_{\varphi_2}}\right\|_{\varphi_1}=1\quad\Rightarrow\quad\varrho_{\varphi_1}\left(\frac{u}{\|v\|_{\varphi_2}}\right)=1=\varrho_{\varphi_2}\left(\frac{v}{\|v\|_{\varphi_2}}\right).\end{equation}
}
\end{remark}

\begin{definition}{\rm 
For $\varphi\in \Phi(\Omega)$, the function $\varphi^*:\Omega\times\mathbb R$ defined as 
$$\varphi^*(x,s)=\sup_{t\ge0}(st-\varphi(x,t))\quad\mbox{for a.a. }x\in\Omega\mbox{ and all }s\in[0,\infty)$$
is called {\it conjugate function of $\varphi$ in the sense of Young}.
}
\end{definition}

\begin{proposition}{\rm(cf.\ \cite[Lemma 2.6.5]{Base})}
Let $\varphi\in \Phi(\Omega)$, then the following H\"older-type inequality holds 
$$\int_\Omega |uv| dx \le 2\|u\|_\varphi\|v\|_{\varphi^*},\,\,\quad\mbox{for all }u\in L^\varphi(\Omega) \mbox{ and }v\in L^{\varphi^*}(\Omega).$$ 
\end{proposition}

\begin{definition}{\rm $\varphi:[0,\infty)\to[0,\infty)$ is called {\it $N$-function} ($N$ stands for {\it nice}) if it is a $\Phi$-function satisfying
$$\lim_{t\to0^+}\frac{\varphi(t)}{t}=0\quad\mbox{and}\quad\lim_{t\to\infty}\frac{\varphi(t)}{t}=\infty.$$
A function $\varphi:\Omega\times\mathbb R\to [0,\infty)$ is said to be a {\it generalized $N$-function}, and is denoted by $\varphi\in N(\Omega)$, if $\varphi(\cdot,t)$ is measurable for all $t\in\mathbb R$ and $\varphi(x,\cdot)$ is an $N$-function for a.a. $x\in\Omega$.}
\end{definition}

\noindent{\bf Remark.} $\varphi\in N(\Omega)$ implies $\varphi^*\in N(\Omega)$.

\begin{definition} 
{\rm Let $\phi,\,\psi\in N(\Omega)$. We say that {\it $\phi$ increases essentially more slowly than $\psi$ near infinity}, and we write $\phi\ll\psi$, if for any $k>0$
$$\lim_{t\to\infty}\frac{\phi(x,kt)}{\psi(x,t)}=0\quad\mbox{uniformly for a.a. }x\in\Omega.$$
}
\end{definition}

\medskip
\noindent
For $\varphi\in \Phi(\Omega)$, the related Sobolev space $W^{1,\varphi}(\Omega)$ is the set of all $L^\varphi(\Omega)$-functions $u$ having $|\nabla u|\in L^\varphi(\Omega)$, and is equipped with the norm $$\|u\|_{1,\varphi}=\|u\|_\varphi+\|\nabla u\|_\varphi,$$
where $\|\nabla u\|_\varphi$ stands for $\|\,|\nabla u|\,\|_\varphi$. Furthermore, if $\varphi\in N(\Omega)$ is locally integrable, we denote by $W^{1,\varphi}_0(\Omega)$ the completion of $C^\infty_0(\Omega)$ in $W^{1,\varphi}(\Omega)$.


\begin{proposition}\label{Wreflexive}{\rm  (cf. \cite[Theorem 10.2]{Musielak}, \cite[Proposition 1.8]{Fan12})} Let $\varphi\in N(\Omega)$ be locally integrable and such that
\begin{equation}\label{phi1}\inf_{x\in\Omega}\varphi(x,1)>0.\end{equation}
Then the spaces $W^{1,\varphi}(\Omega)$ and $W^{1,\varphi}_0(\Omega)$ are Banach spaces which are reflexive if $L^\varphi(\Omega)$ is reflexive. 
\end{proposition}


\subsection{The double phase $N$-function} The function $\mathcal H:\Omega\times[0,\infty)\to[0,\infty)$ defined as $$\mathcal H(x,t):=t^p+a(x)t^q\quad\mbox{for all }(x,t)\in\Omega\times[0,\infty),$$ with $1<p<q$ and $0\le a(\cdot)\in L^1(\Omega)$, 
is a locally integrable, generalized $N$-function satisfying \eqref{phi1} and
$$\mathcal H(x,2t)\le2^q \mathcal H(x,t)\quad\mbox{for a.a. }x\in\Omega\mbox{ and all }t\in[0,\infty),$$ that is condition $(\Delta_2)$.
Therefore, in correspondence to $\mathcal H$, we define the Musielak-Orlicz space $(L^\mathcal H(\Omega),\|\cdot\|_\mathcal H)$ as
$$
\begin{aligned}L^\mathcal H(\Omega)&:=\big\{u:\Omega\to\mathbb R\mbox{ measurable }:\,\varrho_\mathcal H(u) < \infty \big\},\\
\|u\|_\mathcal H&:=\inf\left\{\gamma>0\,:\,\varrho_\mathcal H(u/\gamma)\le 1\right\},
\end{aligned}
$$
where we recall that 
$$\varrho_\mathcal H(u):=\int_\Omega\mathcal H(x,|u|)dx.$$

\begin{remark}\label{brasco}
{\rm In the next lemma we provide an explicit expression for $\|\cdot\|_\mathcal H$. 
	To this aim, for every function $u$ with $a(x)|u|^q\in L^1(\Omega)$ and $\|a^{1/q}u\|_q>0$,  we set
\[
\Theta_{p,q}(u):=\left(\frac{\|a^{1/q}u\|_q}{\|u\|_p}\right)^\frac{q}{q-p}.
\]
We observe that the convex function
\[
\mathcal W(t)=t^p+t^q,\qquad t\in[0,\infty)
\]
is invertible in $[0,\infty)$. 
We have the following result.
\begin{lemma}\label{brasco1}
Let $1<p< q$. Then, for every  $u$ with $a(x)|u|^q\in L^1(\Omega)$ and $\|a^{1/q}u\|_q>0$, there holds
\begin{equation}
\label{norma}
\|u\|_\mathcal H=\frac{\|u\|_p\,\Theta_{p,q}(u)}{\mathcal W^{-1}(\Theta_{p,q}(u)^p)}.
\end{equation}
\end{lemma}
\begin{proof}
Let $\gamma>0$ be admissible for the problem defining $\|u\|_{\mathcal H}$, i.e.
\begin{equation}
\label{condizione}
\frac{1}{\gamma^p}\,\int_\Omega |u|^p\,dx+\frac{1}{\gamma^q}\,\int_\Omega a(x)|u|^q\,dx\le 1.
\end{equation}
We now perform the change of variable
\[
\gamma=\frac{1}{t}\,\left(\int_\Omega |u|^p\,dx\right)^\alpha\,\left(\int_\Omega a(x)|u|^q\,dx\right)^{-\alpha},\qquad t>0,
\]
for some $\alpha\in\mathbb{R}$ that will be chosen later. Then \eqref{condizione} becomes
\[
\begin{split}
t^p\,\left(\int_\Omega |u|^p\,dx\right)^{1-\alpha\,p}&\left(\int_\Omega a(x)|u|^q\,dx\right)^{\alpha\,p}\\
&+t^q\,\left(\int_\Omega |u|^p\,dx\right)^{-\alpha\,q}\,\left(\int_\Omega a(x)|u|^q\,dx\right)^{1+\alpha\,q}\le 1.
\end{split}
\]
If we choose
\[
\alpha:=-\frac{1}{q-p},
\]
the previous inequality becomes
\[
(t^p+t^q)\,\|a^{1/q}u\|_q^{-\frac{pq}{q-p}}\,\|u\|_p^\frac{pq}{q-p}\le 1.
\]
This can be finally rewritten as
\[
\mathcal W(t)\le \Theta_{p,q}(u)^p.
\]
This shows that
\[
\|u\|_\mathcal H=\frac{\displaystyle\left(\int_\Omega a(x)|u|^q\,dx\right)^{\frac{1}{q-p}}}{\displaystyle\left(\int_\Omega |u|^p\,dx\right)^{\frac{1}{q-p}}}\,\inf\left\{\frac1t>0\, :\, \mathcal W(t)\le \Theta_{p,q}(u)^p\right\}.
\]
By using that $\mathcal W$ is strictly monotonically increasing in $[0,\infty)$, we get the expression \eqref{norma}.
\end{proof}
}
\end{remark}

\noindent
We recall here the following definition.
\begin{definition} A function $\varphi\in N(\Omega)$ is {\it uniformly convex} if for every $\varepsilon>0$ there exists $\delta>0$ such that
$$|t-s|\le\varepsilon\max\{t,s\}\quad\mbox{or}\quad\varphi\left(x,\dfrac{t+s}2\right)\le (1-\delta)\frac{\varphi(x,t)+\varphi(x,s)}2$$
for all $t,s\ge0$ and a.a. $x\in\Omega$. 
\end{definition} 
\smallskip

\noindent 
We endow the spaces $W^{1,\mathcal H}(\Omega)$ and $W^{1,\mathcal H}_0(\Omega)$ with the norm 
$$\|u\|_{1,\mathcal H}:=\|u\|_\mathcal H+\|\nabla u\|_\mathcal H.$$

\begin{proposition}\label{reflexivity} The spaces $L^\mathcal H(\Omega)$, $W^{1,\mathcal H}(\Omega)$ and $W^{1,\mathcal H}_0(\Omega)$ are uniformly convex, and so reflexive, Banach spaces. 
\end{proposition} 
\begin{proof} By Propositions \ref{Banach} and \ref{Wreflexive}, $L^\mathcal H(\Omega)$, $W^{1,\mathcal H}(\Omega)$, and $W^{1,\mathcal H}_0(\Omega)$ are complete. For the second part of the thesis it suffices to prove that $L^\mathcal H(\Omega)$ is reflexive. Since by Propostion \ref{Banach}, $L^\mathcal H(\Omega)$ is a Banach space, if we prove that $L^\mathcal H(\Omega)$ is uniformly convex, the reflexivity follows by the Milman-Pettis theorem. 
By Theorems~2.4.11 and 2.4.14 of \cite{Base}, in order to prove that $L^\mathcal H(\Omega)$ is uniformly convex, it is enough to show that the $N$-function $\mathcal H$ is uniformly convex. Let $\varepsilon>0$ and $t,s\ge0$ be such that $|t-s|>\varepsilon\max\{t,s\}$. By Remark 2.4.16 of \cite{Base} there exist $\delta_p(\varepsilon),\delta_q(\varepsilon)>0$ such that
$$\left(\dfrac{t+s}2\right)^p\le(1-\delta_p(\varepsilon))\frac{t^p+s^p}2\quad\mbox{and}\quad\left(\dfrac{t+s}2\right)^q\le(1-\delta_q(\varepsilon))\frac{t^q+s^q}2,$$
thus 
$$\left(\dfrac{t+s}2\right)^p+a(x)\left(\dfrac{t+s}2\right)^q\le(1-\min\{\delta_p(\varepsilon),\delta_q(\varepsilon)\})\frac{t^p+a(x)t^q+s^p+a(x)s^q}2.$$
This concludes the proof.
\end{proof}

\noindent
In the following, the notation $X\hookrightarrow Y$ means that the space $X$ is {\em continuously} embedded into the space $Y$,
while $X\hookrightarrow\hookrightarrow Y$ means that $X$ is {\em compactly} embedded into $Y$.

\begin{proposition}[Embeddings, I]
	\label{embeddingW} Put $p^*:=np/(n-p)$ if $p<n$, $p^*:=+\infty$ otherwise, and 
$$
L^q_a(\Omega):=\left\{u:\Omega\to\mathbb R\,\mbox{{\rm measurable}}\,:\,\int_\Omega a(x)|u|^qdx<\infty\right\},
$$ 
endowed with the norm 
$$
\|u\|_{q,a}:=\left(\int_\Omega a(x)|u|^qdx\right)^{1/q}.
$$ 
Then the following embeddings hold:
\begin{itemize} 
\item[$(i)$] $L^\mathcal H(\Omega)\hookrightarrow L^r(\Omega)$ and $W^{1,\mathcal H}_0(\Omega)\hookrightarrow W^{1,r}_0(\Omega)$ for all $r\in [1, p]$;
\item[$(ii)$] if $p\neq n$, then
$$
\text{$W_0^{1,\mathcal H}(\Omega)\hookrightarrow L^r(\Omega)$,\,\,\, for all $r\in [1, p^*]$;}
$$ 
if $p=n$, then
$$
\text{$W_0^{1,\mathcal H}(\Omega)\hookrightarrow L^r(\Omega)$,\,\,\, for all $r\in [1, \infty)$;} 
$$
\item[$(iii)$] if $p\le n$, then
$$
\text{$W_0^{1,\mathcal H}(\Omega)\hookrightarrow\hookrightarrow L^r(\Omega)$ for all $r\in [1, p^*)$;} 
$$
if $p>n$, then
$$
\text{$W_0^{1,\mathcal H}(\Omega)\hookrightarrow\hookrightarrow L^\infty(\Omega)$;}
$$ 
\item[$(iv)$] $L^\mathcal H(\Omega)\hookrightarrow L^q_a(\Omega)$;
\item[$(v)$] if $a\in L^\infty(\Omega)$, then
$L^q(\Omega)\hookrightarrow L^\mathcal H(\Omega)$. 
\end{itemize}
\end{proposition}
\begin{proof} Put $\mathcal H_p(x,t):=t^p$ for all $t\ge0$ and $x\in\Omega$. Clearly, $\mathcal H_p\preceq\mathcal H$, hence by Proposition~\ref{embeddingL}, $L^{\mathcal H}(\Omega)\hookrightarrow L^{p}(\Omega)$ and  $W_0^{1,\mathcal H}(\Omega)\hookrightarrow W_0^{1,p}(\Omega)$. Therefore, $(i)$ follows by the boundedness of $\Omega$. While, $(ii)$ and $(iii)$ follow by the embedding results on classical Lebesgue and Sobolev spaces.
Now, let $u\in L^\mathcal H(\Omega)$, then 
$$\int_\Omega a(x)|u|^qdx\le\int_\Omega(|u|^p+ a(x)|u|^q)dx=\varrho_\mathcal H(u).$$
Therefore, if $u\neq 0$,
$$
\int_\Omega a(x)\left(\frac{|u|}{\|u\|_\mathcal H}\right)^q dx\le 1,
$$
and so we conclude
$$
\|u\|_{q,a}\le\|u\|_\mathcal H,
$$
which proves $(iv)$.
Finally, for all $t\ge0$ and a.a. $x\in\Omega$, if $a\in L^\infty(\Omega)$, we have
$$\mathcal H(x,t)\le(1+t^q)+a(x)t^q\le1+(1+\|a\|_\infty)t^q,$$
so $(v)$ follows once again by Proposition \ref{embeddingL}.    
\end{proof}

\noindent
Although we will not use it explicitly, the next lemma could be useful in some situations.

\begin{lemma}\label{4domconv} 
	Let $(u_h)$ be a sequence in $L^\mathcal H(\Omega)$ such that $u_h\to u\in L^\mathcal H(\Omega)$. Then, there exist a subsequence $(u_{h_j})$ and a function $v\in L^\mathcal H(\Omega)$ such that 
	$$\begin{gathered}u_{h_j}\to u\quad\mbox{a.e. in }\Omega,\\
	|u_{h_j}|\le v\quad\mbox{for all $j$, a.e. in }\Omega.\end{gathered}$$
\end{lemma}
\begin{proof} By Proposition \ref{embeddingW} $L^\mathcal H(\Omega)\hookrightarrow L^p(\Omega)$, hence up to a subsequence $u_h\to u$ a.e. in $\Omega$. We prove now the second part of the statement. Since $(u_h)$ is a Cauchy sequence in $L^\mathcal H(\Omega)$, we can find a subsequence $(u_{h_j})$ for which $\|u_{h_{j+1}}-u_{h_j}\|_\mathcal H\le 2^{-j}$. For all $m\ge 1$ we define 
	$$f_m(x):=\sum_{j=1}^m|u_{h_{j+1}}(x)-u_{h_j}(x)|$$
	which satisfies $$\|f_m\|_\mathcal H\le\sum_{j=1}^m\|u_{h_{j+1}}-u_{h_j}\|_\mathcal H\le 1\quad\mbox{for all }m\in\mathbb N$$
	and by the unit ball property $$\varrho_\mathcal H(f_m)\le 1\quad\mbox{for all }m\in\mathbb N.$$ 
	Clearly, $f_{m}(x)\le f_{m+1}(x)$ for all $m$, a.e. in $\Omega$, and so, by the monotone convergence theorem, there exists $f\in L^1(\Omega)$ such that $f_m\to f$ a.e. in $\Omega$. By virtue of Lemma 2.3.16-$(b)$ of \cite{Base}, $$\varrho_\mathcal H(f)=\lim_{m\to\infty}\varrho_\mathcal H(f_m),$$
	whence $f\in L^\mathcal H(\Omega)$. Now, for all $m>\ell\ge2$ and for a.a. $x\in\Omega$
	\begin{equation}\label{dis}|u_{h_m}(x)-u_{h_\ell}(x)|\le|u_{h_m}(x)-u_{h_{m-1}}(x)|+\dots+|u_{h_{\ell+1}}(x)-u_{h_\ell}(x)|=f_{m-1}(x)-f_{\ell-1}(x).\end{equation}This implies that for a.a. $x\in\Omega$, $(u_{h_j}(x))$ is a Cauchy sequence in $\mathbb R$ and so it converges to some $\bar u(x)\in\mathbb R$. Furthermore,  
	passing to the limit for $m\to\infty$ in \eqref{dis}, we get for all $j\ge 2$ and for a.a. $x\in\Omega$
	\begin{equation}\label{almostfin}|\bar u(x)-u_{h_j}(x)|\le f(x).
	\end{equation}
	Thus, by Lemma 2.3.16 $(c)$ of \cite{Base}, we obtain  
	$u_{h_j}\to\bar u$ in $L^\mathcal H(\Omega),$
	and so $u=\bar u$ a.e. in $\Omega$. Finally, \eqref{almostfin} yields
	$$|u_{h_j}(x)|\le |u(x)|+f(x)\quad\mbox{for a.a. }x\in\Omega,$$
	and the proof is concluded by taking $v:=|u|+f\in L^\mathcal H(\Omega)$.
\end{proof}


\noindent From now on in the paper, unless explicitly stated, we shall assume that
$$
1<p<q<n.
$$ 

\begin{definition}\label{criticalH} 
{\rm For all $x\in\Omega$ denote by $\mathcal H^{-1}(x,\cdot):[0,\infty)\to[0,\infty)$ the inverse function of $\mathcal H(x,\cdot)$ and define $\mathcal H^{-1}_*:\Omega\times[0,\infty)\to[0,\infty)$ by 
$$\mathcal H^{-1}_*(x,s)=\int_0^s\frac{\mathcal H^{-1}(x,\tau)}{\tau^{(n+1)/n}}d\tau\quad\mbox{for all }(x,s)\in\Omega\times[0,\infty).$$
The function $\mathcal H_*:(x,t)\in \Omega\times[0,\infty)\mapsto s\in[0,\infty)$ such that $\mathcal H^{-1}_*(x,s)=t$ is called {\it Sobolev conjugate function of} $\mathcal H$.}
\end{definition}

\begin{proposition}[Embeddings, II]\label{sob-poinc} Assume that \eqref{cond-pq} holds.
Then the following facts hold.
\begin{itemize}
\item[$(i)$] $W^{1,\mathcal H}(\Omega)\hookrightarrow L^{\mathcal H_*}(\Omega)$.
\item[$(ii)$] If $\mathcal K\in N(\Omega)$, $\mathcal K:\Omega\times[0,\infty)\to[0,\infty)$ is continuous and such that $\mathcal K\ll\mathcal H_*$, then 
$$
\text{$W^{1,\mathcal H}(\Omega)\hookrightarrow\hookrightarrow L^\mathcal K(\Omega)$.}
$$
\item[$(iii)$] $\mathcal H\ll\mathcal H_*$, and consequently 
$$
\text{$W^{1,\mathcal H}(\Omega)\hookrightarrow\hookrightarrow L^\mathcal H(\Omega)$.}
$$
\item[$(iv)$] The following Poincar\'e-type inequality holds
\begin{equation}\label{poinc}\|u\|_\mathcal H\le C \|\nabla u\|_\mathcal H,\,
\quad\mbox{for all }u\in W^{1,\mathcal H}_0(\Omega),\end{equation}
for some constant $C>0$ independent of $u$.
\end{itemize} 
\end{proposition} 
\begin{proof} We refer to Theorems 1.1 and 1.2 of \cite{FanImbedding}. It suffices to prove condition $(2)$ of Proposition~3.1 of \cite{FanImbedding}, i.e. that there exist three positive constants $\delta<1/n$, $c_0$ and $t_0$ such that 
\begin{equation}\label{p5}\left|\frac{\partial \mathcal H(x,t)}{\partial x_j}\right|\le c_0(\mathcal H(x,t))^{1+\delta}\end{equation}
for all $j=1,\dots,n$, $x\in\Omega$ for which $\nabla a(x)$ exists, and $t\ge t_0$.
If we put $\delta:=q/p-1$ and $c_a>0$ the Lipschitz constant of $a$, we get 
$$\left|\frac{\partial \mathcal H(x,t)}{\partial x_j}\right|\le c_at^q\le c_a (t^p+a(x)t^q)^{q/p}\quad\mbox{for a.a. } x\in\Omega\mbox{, all }t>0\mbox{, and }j=1,\dots,n,$$
that is \eqref{p5} with $c_0:=c_a$, $\delta:=q/p-1<1/n$ and any $t_0>0$.
\end{proof}

\begin{remark}\label{remh'}{\rm Poincar\'e-type inequality \eqref{poinc} has been proved also in \cite{HHK} under the more general assumption  
\begin{equation}\label{cond-reg}\Omega \mbox{ is quasi-convex and }a\in C^{0,\alpha}(\overline{\Omega})\mbox{, with }\dfrac{q}p\le 1+\dfrac{\alpha}n\quad\text{for some $\alpha\in (0,1]$.}\end{equation}
We wish to stress that the bound on $q/p$ given in \eqref{cond-reg}, was required for the first time in the papers \cite{BarColMin1,BarColMin2} dealing with regularity of local minimizers for double phase variational integrals.
Furthermore, we observe that, since $p^*>p(1+1/n)$, 
both \eqref{cond-pq} and \eqref{cond-reg} imply $q<p^*$.}
\end{remark}

\noindent
As a consequence of inequality \eqref{poinc}, if either assumption \eqref{cond-pq} or assumption \eqref{cond-reg} holds, we equip the space $W^{1,\mathcal H}_0(\Omega)$ with the equivalent norm
$$\|\nabla u\|_\mathcal H.$$

\section{The eigenvalue problem}\label{sec3}
\subsection{Derivation of the Euler-Lagrange equation}
Put
$$
K(u):=\|\nabla u\|_\mathcal H,\quad
k(u):=\|u\|_\mathcal H\,\,\, \text{for $u\in W^{1,\mathcal H}_0(\Omega)$},\quad\,\,
\mathcal M:=\big\{u\in W^{1,\mathcal H}_0(\Omega)\,:\, k(u)=1\big\}.
$$
Let us consider the Rayleigh ratio 
\begin{equation}\label{R}\frac{K(u)}{k(u)}=\frac{\|\nabla u\|_\mathcal H}{\|u\|_\mathcal H}\end{equation} 
and define the {\it first eigenvalue} as
\begin{equation}
\label{Rr}
\lambda_\mathcal H^{1}:=\inf_{u\in W^{1,\mathcal H}_0(\Omega)\setminus\{0\}}\frac{\|\nabla u\|_\mathcal H}{\|u\|_\mathcal H}=\inf_{u\in\mathcal M}K(u).\end{equation}
{\it We claim that the following equation 
\begin{equation}
	 \label{EL}
	 \begin{aligned}-\mathrm{div}&\left(\left[p\left(\frac{|\nabla u|}{K(u)}\right)^{p-2}+qa(x)\left(\frac{|\nabla u|}{K(u)}\right)^{q-2}\right]\frac{\nabla u}{K(u)}\right)\\
	 &\hspace{1.5cm}=\lambda S(u) \left[p\left(\frac{|u|}{k(u)}\right)^{p-2}+qa(x)\left(\frac{|u|}{k(u)}\right)^{q-2}\right]\frac u{k(u)}, \quad\; u\in W^{1,\mathcal H}_0(\Omega)\setminus\{0\}
	 \end{aligned}
	 \end{equation}
	  is the Euler-Lagrange equation corresponding to the minimization of the Rayleigh ratio \eqref{R}.} In \eqref{EL} we have denoted by $S(u)$ the following quantity
\begin{equation}\label{Su}S(u):=\frac{\displaystyle\int_\Omega\left[p\left(\frac{|\nabla u|}{K(u)}\right)^p+qa(x)\left(\frac{|\nabla u|}{K(u)}\right)^q\right]dx}{\displaystyle\int_\Omega\left[p\left(\frac{|u|}{k(u)}\right)^p+qa(x)\left(\frac{|u|}{k(u)}\right)^q\right]dx}.\end{equation}
We note that equation \eqref{EL} reduces to \eqref{ELM} if $u\in\mathcal M$, since in that case $K(u)=\lambda$, $k(u)=1$ and $S(u)$ reads as in \eqref{SuM}. 

\noindent
In order to prove the claim, we define for all $u\in W^{1,\mathcal H}_0(\Omega)\setminus\{0\}$ and $v\in W^{1,\mathcal H}_0(\Omega)$:
\begin{align*} 
\langle A(u),v\rangle:&=\frac{\displaystyle\int_\Omega\left[p\left(\frac{|u|}{k(u)}\right)^{p-2}+qa(x)\left(\frac{|u|}{k(u)}\right)^{q-2}\right]\frac{u}{k(u)}v\,dx}{\displaystyle\int_\Omega\left[p\left(\frac{|u|}{k(u)}\right)^p+qa(x)\left(\frac{|u|}{k(u)}\right)^q\right]dx},\\
\langle B(u),v\rangle:&=\frac{\displaystyle\int_\Omega\left[p\left(\frac{|\nabla u|}{K(u)}\right)^{p-2}+qa(x)\left(\frac{|\nabla u|}{K(u)}\right)^{q-2}\right]\frac{\nabla u}{K(u)}\cdot \nabla v\,dx}{\displaystyle\int_\Omega\left[p\left(\frac{|\nabla u|}{K(u)}\right)^p+qa(x)\left(\frac{|\nabla u|}{K(u)}\right)^q\right]dx}.
\end{align*}

\begin{proposition}\label{kC1} 
$k\in C^1(W^{1,\mathcal H}_0(\Omega)\setminus\{0\})$ with $k'(u)=A(u)$ for all $u\in W^{1,\mathcal H}_0(\Omega)\setminus\{0\}$.
\end{proposition}
\begin{proof} Reasoning as in Lemma A.1 of \cite{Franzi}, we get 
$$\lim_{\varepsilon\to 0^+}\frac{k(u+\varepsilon v)-k(u)}{\varepsilon}=\langle A(u),v\rangle.$$
Put 
$$f(v):=\int_\Omega\left[p\left(\frac{|v|}{k(v)}\right)^p+qa(x)\left(\frac{|v|}{k(v)}\right)^q\right]dx\quad\mbox{for all }v\in W^{1,\mathcal H}_0(\Omega)\setminus\{0\}.$$
We observe that, being $1<p<q$ and by the unit ball property,
\begin{equation}\label{den>1}f(v)\ge \varrho_\mathcal H\left(\frac{v}{k(v)}\right)=1.\end{equation}
By H\"older's inequality and by Proposition \ref{embeddingW}, it results that 
$$\begin{aligned}|\langle A(u),v\rangle|&\le\frac{p}{k(u)^{p-1}}\int_\Omega|u|^{p-1}|v|dx+\frac{q}{k(u)^{q-1}}\int_\Omega a^{1/q'}|u|^{q-1}a^{1/q}|v|dx\\
&\le\frac{p}{k(u)^{p-1}}\|u\|_p^{p-1}\|v\|_p+\frac{q}{k(u)^{q-1}}\|u\|_{q,a}^{q-1}\|v\|_{q,a}\\
&\le\left(\frac{\mathcal S_p p}{k(u)^{p-1}}\|u\|_p^{p-1}+\frac{\mathcal S_{q,a} q}{k(u)^{q-1}}\|u\|_{q,a}^{q-1}\right)\|v\|_{1,\mathcal H},\end{aligned}
$$
where $\mathcal S_p,\,\mathcal S_{q,a}>0$ are the Sobolev constants for the embeddings of $W^{1,\mathcal H}_0(\Omega)$ in $L^p(\Omega)$ and in $L^q_a(\Omega)$, respectively. Therefore, $A(u)$ belongs to the dual space $(W^{1,\mathcal H}_0(\Omega))'$ of $W^{1,\mathcal H}_0(\Omega)$ and $k$ is G\^ateaux differentiable in $u$, with $k'(u)=A(u)$ for all $u\in W^{1,\mathcal H}_0(\Omega)\setminus\{0\}$.
It remains to prove that $k':W^{1,\mathcal H}_0(\Omega)\to(W^{1,\mathcal H}_0(\Omega))'$ is continuous, i.e. that given a sequence $(u_h)\subset W^{1,\mathcal H}_0(\Omega)\setminus\{0\}$ such that $u_h\to u\neq 0$ in $W^{1,\mathcal H}_0(\Omega)$,
$$\sup_{\underset{\|v\|_{1,\mathcal H}\le1}{v\in W^{1,\mathcal H}_0(\Omega)}}|\langle k'(u_h)-k'(u),v\rangle |\to0 \quad\mbox{as }h\to\infty.$$ 
For all $v\in W^{1,\mathcal H}_0(\Omega)$ with $\|v\|_{1,\mathcal H}\le1$ we get 
$$\begin{aligned}|\langle k'(u_h)-k'(u),v\rangle |&\le\left|\frac{\displaystyle\int_\Omega p\left(\frac{|u_h|}{k(u_h)}\right)^{p-2}\frac{u_h}{k(u_h)}v\,dx}{f(u_h)}-\frac{\displaystyle\int_\Omega p\left(\frac{|u|}{k(u)}\right)^{p-2}\frac{u}{k(u)}v\,dx}{f(u)}\right|\\
&\phantom{=}+\left|\frac{\displaystyle\int_\Omega q a(x)\left(\frac{|u_h|}{k(u_h)}\right)^{q-2}\frac{u_h}{k(u_h)}v\,dx}{f(u_h)}-\frac{\displaystyle\int_\Omega qa(x)\left(\frac{|u|}{k(u)}\right)^{q-2}\frac{u}{k(u)}v\,dx}{f(u)}\right|\\
&\le pI_1^{(h)}+q I_2^{(h)},\end{aligned}$$ 
where 
$$\begin{aligned}I_1^{(h)}&:=\int_\Omega |v|\left|\frac1{f(u_h)}\left(\frac{|u_h|}{k(u_h)}\right)^{p-2}\frac{u_h}{k(u_h)}-\frac1{f(u)}\left(\frac{|u|}{k(u)}\right)^{p-2}\frac{u}{k(u)}\right|dx\\
I_2^{(h)}&:=\int_\Omega a(x)|v|\left|\frac1{f(u_h)}\left(\frac{|u_h|}{k(u_h)}\right)^{q-2}\frac{u_h}{k(u_h)}-\frac1{f(u)}\left(\frac{|u|}{k(u)}\right)^{q-2}\frac{u}{k(u)}\right|dx.\end{aligned}$$
We show that $I_2^{(h)}\to 0$ as $h\to\infty$, the proof for $I_1^{(h)}$ is identical, with $a(\cdot)$ replaced by the constant function 1 and $q$ replaced by $p$. We get by H\"older's inequality
$$\begin{aligned}I_2^{(h)}&\le\int_\Omega \frac{a(x)^{1/q}|v|}{f(u_h)}a(x)^{1/q'}\left|\left(\frac{|u_h|}{k(u_h)}\right)^{q-2}\frac{|u_h|}{k(u_h)}-\left(\frac{|u|}{k(u)}\right)^{q-2}\frac{|u|}{k(u)}\right|dx\\
&\phantom{=}+\int_\Omega a(x)^{1/q}|v|a(x)^{1/q'}\left(\frac{|u|}{k(u)}\right)^{q-1}\left|\frac1{f(u_h)}-\frac1{f(u)}\right| dx\\
&\le \|v\|_{q,a}\left(\int_\Omega a(x)\left|\left(\frac{|u_h|}{k(u_h)}\right)^{q-2}\frac{|u_h|}{k(u_h)}-\left(\frac{|u|}{k(u)}\right)^{q-2}\frac{|u|}{k(u)}\right|^{q'}dx\right)^{1/q'}\\
&\phantom{=}+\|v\|_{q,a}\left|\frac1{f(u_h)}-\frac1{f(u)}\right|\left(\int_\Omega a(x)\left(\frac{|u|}{k(u)}\right)^{q} dx\right)^{1/q'}\\
&\le\mathcal S_{q,a}(\mathcal I^{(h)}+\mathcal J^{(h)}),\end{aligned}$$
where 
$$\begin{aligned}\mathcal I^{(h)}&:=\left(\int_\Omega a(x)\left|\left(\frac{|u_h|}{k(u_h)}\right)^{q-2}\frac{|u_h|}{k(u_h)}-\left(\frac{|u|}{k(u)}\right)^{q-2}\frac{|u|}{k(u)}\right|^{q'}dx\right)^{1/q'}\\
\mathcal J^{(h)}&:=\left|\frac1{f(u_h)}-\frac1{f(u)}\right|\left(\int_\Omega a(x)\left(\frac{|u|}{k(u)}\right)^{q} dx\right)^{1/q'}.
\end{aligned}$$
First, we estimate $\mathcal I^{(h)}$. We pick a subsequence $(u_{h_j})$. Since by Proposition \ref{embeddingW} $W^{1,\mathcal H}_0(\Omega)\hookrightarrow L^q_a(\Omega)$, thus $u_{h_j}\to u$ in $L^q_a(\Omega)$, i.e. $a^{1/q}u_{h_j}\to a^{1/q}u$ in $L^q(\Omega)$. Up to a subsequence, $u_{h_j}\to u$ a.e. in $\Omega$ and there exists a function $w\in L^q(\Omega)$ such that $a^{1/q}|u_{h_j}|\le w$ a.e. in $\Omega$ for all $j$. Therefore, for $j$ sufficiently large and for $\varepsilon\in(0,k(u))$,
$$\begin{aligned}a(x)\left|\left(\frac{|u_{h_j}|}{k(u_{h_j})}\right)^{q-2}\frac{u_{h_j}}{k(u_{h_j})}-\left(\frac{|u|}{k(u)}\right)^{q-2}\frac{u}{k(u)}\right|^{q'}&\le 2^{q'-1}a(x)\left[\left(\frac{|u_{h_j}|}{k(u_{h_j})}\right)^q+\left(\frac{|u|}{k(u)}\right)^q\right]\\
&\le 2^{q'-1}\left[\left(\frac{w}{k(u_{h_j})}\right)^q+\left(\frac{w}{k(u)}\right)^q\right]\\
&\le 2^{q'-1}\left[\frac1{(k(u)-\varepsilon)^q}+\frac1{k(u)^q}\right]w^q\in L^1(\Omega), 
\end{aligned}$$
where in the last inequality we used the fact that $k(u_{h_j})\to k(u)\neq 0$. Consequently, by the dominated convergence theorem, $\mathcal I^{(h_j)}\to 0$ as $j\to\infty$ and by the arbitrariness of the subsequence, $\mathcal I^{(h)}\to 0$ as $h\to\infty$. Now, in order to prove that also $\mathcal J^{(h)}\to 0$, it is enough to prove that $f(u_h)\to f(u)$ as $h\to\infty$. By the facts that $W^{1,\mathcal H}_0(\Omega)\hookrightarrow L^p(\Omega)$ and $W^{1,\mathcal H}_0(\Omega)\hookrightarrow L^q_a(\Omega)$, and by the dominated convergence theorem, we easily get  $$\begin{gathered}\lim_{h\to\infty}\int_\Omega\left(\frac{|u_h|}{k(u_h)}\right)^pdx=\int_\Omega\left(\frac{|u|}{k(u)}\right)^pdx\\
\lim_{h\to\infty}\int_\Omega a(x)\left(\frac{|u_h|}{k(u_h)}\right)^qdx=\int_\Omega a(x)\left(\frac{|u|}{k(u)}\right)^qdx.\end{gathered}$$
This proves that $f(u_h)\to f(u)$ and so $\mathcal J^{(h)}\to 0$. 
Therefore, $(I_2^{(h)})$ converges to zero as $h\to\infty$ and the proof is concluded by the arbitrariness of $v$.
\end{proof}

\noindent
As a consequence of the last proposition, $\mathcal M$ is a $C^1$ Banach manifold. 
By using analogous techniques as in the proof of Proposition \ref{kC1}, it is possible to prove the following result.

\begin{proposition}\label{KC1} 
$K\in C^1(W^{1,\mathcal H}_0(\Omega)\setminus\{0\})$ with $K'(u)=B(u)$ for all $u\in W^{1,\mathcal H}_0(\Omega)\setminus\{0\}$.
\end{proposition}

\noindent
Reasoning as in Section 3 of \cite{Franzi}, we find that a necessary condition for minimality of the Rayleigh ratio \eqref{R} is 
$$
\frac{\langle K'(u),v\rangle}{K(u)}=\frac{\langle k'(u),v\rangle}{k(u)}\,\quad\mbox{for all }u, v\in W^{1,\mathcal H}_0(\Omega), \,u\neq 0.
$$
Together with Propositions \ref{kC1} and \ref{KC1}, this yields the claim with $\lambda=K(u)/k(u)$, and justifies the following definition. 

\begin{definition}\label{eigen}{\rm We say that $u\in W^{1,\mathcal H}_0(\Omega)\setminus\{0\}$ is an {\it eigenfunction} of \eqref{EL} if 
\begin{equation}\label{weak}\begin{aligned}\int_\Omega\left(p\left|\frac{\nabla u}{K(u)}\right|^{p-2}+q a(x)\left|\frac{\nabla u}{K(u)}\right|^{q-2}\right)&\frac{\nabla u}{K(u)}\cdot\nabla v\,dx \\
&\hspace{-1.5cm}=\lambda S(u)\int_\Omega\left(p\left|\frac{u}{k(u)}\right|^{p-2}+qa(x)\left|\frac{u}{k(u)}\right|^{q-2}\right)\frac{u}{k(u)} v\,dx\end{aligned}\end{equation}
for all $v\in C^\infty_0(\Omega)$, where $S(u)$ is defined as in \eqref{Su}. The real number $\lambda$ is the corresponding {\it eigenvalue}.
The set 
$$
\Lambda:=\big\{\lambda\in\mathbb R\,:\,\lambda \mbox{ is eigenvalue of \eqref{EL}}\big\},
$$ 
is called {\it spectrum}.} 
\end{definition}

\noindent
We remark here that, by a standard density argument, we can take 
any $v\in W^{1,\mathcal H}_0(\Omega)$ as test function in \eqref{weak}.
Testing equation \eqref{weak} with $v=u$ yields 
\begin{equation}
\label{legame-lambda}
\lambda=\frac{K(u)}{k(u)},
\end{equation}
so that if $u\in {\mathcal M}$, then $\lambda=\|\nabla u\|_{\mathcal H}$ and equation \eqref{ELM} holds.

\subsection{$L^\infty$-bound of eigenfunctions}
\label{boundd-L}
\noindent
Assume that condition \eqref{cond-pq} on $p,q,a$ and $\Omega$ holds.
If $u\in \mathcal M$ is a weak solution to the Euler-Lagrange equation
	 \begin{equation}
	 \label{ELM-2}
	 \begin{aligned}-\mathrm{div}&\left(p\left|\frac{\nabla u}{\lambda}\right|^{p-2}\frac{\nabla u}{\lambda}+qa(x)\left|\frac{\nabla u}{\lambda}\right|^{q-2}\frac{\nabla u}{\lambda}\right)=\lambda S(u)(p|u|^{p-2}u+qa(x)|u|^{q-2}u),
	 \end{aligned}
	 \end{equation}
	 with $\lambda=\|\nabla u\|_{\mathcal H}>0$ and $S(u)$ is as in formula 
	 \eqref{SuM}, then by following a standard
	 argument it is possible to prove that there exists a positive constant 
	 $C(n,p,\lambda,a)$ such that
	 $$
	 \|u\|_{L^\infty(\Omega)}\leq C(n,p,\lambda,a).
	 $$
	 Observe that $S(u)\leq q$ from formula \eqref{SuM}  and $\lambda=\|\nabla u\|_{\mathcal H}$.
	 For all $t\ge 0$ and $k>0$, we set $t_k:=\min\{t,k\}.$
	For all $r\ge 2$, $k>0$, the mapping $t\mapsto t|t|_k^{r-2}$ is Lipschitz continuous in $\R$, hence $v=u|u|_k^{r-2}\in W^{1,\mathcal H}_0(\Omega)$ and we have
	$$
	\nabla v=(r-1)|u|_k^{r-2} \nabla u, \,\,\,\,\text{if $|u|\leq k$},\quad\,\,\, 
	\nabla v=|u|_k^{r-2} \nabla u, \,\,\,\,\text{if $|u|\geq k$.}
	$$
	We choose it as a test function, getting (recall that $|s|_k\leq |s|$ for any $s$ and $k$),
\begin{align}
\label{weak2}
 (r-1)&\int_\Omega\left(p\left|\frac{\nabla u}{\lambda}\right|^{p}|u|_k^{r-2}+q a(x)\left|\frac{\nabla u}{\lambda}\right|^{q}|u|_k^{r-2}\right)\chi_{\{|u|\leq k\}}  dx \\
& \leq S(u)\int_\Omega\left(p\left|u\right|^{p}+qa(x)\left|u\right|^{q}\right)|u|_k^{r-2}dx  \notag \\
& \leq C'\int_\Omega(|u|^{p+r-2}+\left|u\right|^{q+r-2})dx,  \notag
\end{align}
for some positive constant $C'$ which depends only on $p,q,a$.\
Taking into account that $a\geq 0$  and using 
Fatou's lemma (by letting $k\to\infty$), it follows that
\begin{equation}\label{weak3}
(r-1)\int_\Omega |\nabla u|^{p}|u|^{r-2} dx \leq C\int_\Omega
(|u|^{p+r-2}+\left|u\right|^{q+r-2})dx,
\end{equation}
with $C$ depending on $p,q,a$ and $\lambda$. In light of condition \eqref{cond-pq}, we know that $q<p^*$. But then this is exactly
the estimate that is usually obtained to get the $L^m$-estimate for any $m\geq 1$ for the
$p$-Laplacian problem $-\Delta_p u=f(u)$ in $\Omega$ and $u=0$ on $\partial\Omega$, for a subcritical nonlinearity $f:\R\to\R$ which satisfies the growth
condition
$$
|f(s)|\leq C|s|^{p-1}+C|s|^{q-1}\quad \text{for all $s\in\R$},\quad 1<q<p^*.
$$
For the explicit computations following inequality \eqref{weak3} and the bootstrap
argument yielding $u\in L^m(\Omega)$ for every $m\geq 1$, one can argue e.g.\
as in \cite{You}. Then similar bootstrap arguments allow to prove the $L^\infty$-
estimate.

\subsection{On the first eigenvalue} 

Throughout this subsection we shall assume the validity of condition \eqref{cond-pq} 
and we endow $W^{1,\mathcal H}_0(\Omega)$ with the $L^\mathcal H$-norm of the gradient. 
First of all, we seek ground states, i.e.\ least energy solutions, of \eqref{EL}.
In particular, the ground states of \eqref{EL} are the minimizers of $K|_\mathcal M$ and the corresponding energy level is the first eigenvalue
$\lambda_\mathcal H^{1}$. Minimizers $u$, if they exist, must satisfy the Euler-Lagrange equation \eqref{weak}, obtained for minimizers
of the quotients $\|\nabla v\|_{\mathcal H}/\|v\|_{\mathcal H}$ among nonzero functions. 
Since $u\in {\mathcal M}$, then equation \eqref{weak} can be reduced to
\eqref{ELM} since it turns out that $\|\nabla u\|_{\mathcal H}=\lambda$ by formula \eqref{legame-lambda}.
After giving the proof of Theorem~\ref{main1}, we shall collect some properties 
of the first eigenvalue.
\medskip

\noindent
$\bullet$ {\it {Proof of Theorem \ref{main1}}.} 
By the Poincar\'e-type inequality \eqref{poinc}, there exists $C>0$ independent of $v$ for which  
$$\frac{\|\nabla v\|_\mathcal H}{\|v\|_\mathcal H}\ge \frac1C$$
for all $v\in W^{1,\mathcal H}_0(\Omega)\setminus\{0\}$. Thus $\lambda_\mathcal H^{1}>0$. Now, let $(v_h)\subset \mathcal M$ be such that
$$\lim_{h\to\infty}\|\nabla v_h\|_\mathcal H=\lambda_\mathcal H^{1}.$$
Clearly, $(v_h)$ is bounded in the reflexive Banach space $W^{1,\mathcal H}_0(\Omega)$ and so we can extract a subsequence $(v_{h_j})$ weakly converging to $u$ in $W^{1,\mathcal H}_0(\Omega)$. Therefore, by Proposition \ref{sob-poinc}-$(iii)$, $\|v_{h_j}\|_\mathcal H\to\|u\|_\mathcal H$ as $j\to\infty$, and so $\|u\|_\mathcal H=1$. Since the norm is weakly lower semicontinuous, 
$$\|\nabla u\|_\mathcal H\le\liminf_{j\to\infty}\|\nabla v_{h_j}\|_\mathcal H=\lambda_\mathcal H^{1}.$$
This proves that $u$ is a minimizer of \eqref{Rr}. Clearly also $|u|\ge0$ is a minimizer, so we may assume $u\geq 0$ a.e.\
Also, the Euler-Lagrange equation \eqref{ELM} is satisfied. Taking into account that $S(u)\geq 0$, we get
\begin{equation}\label{weak4}
\int_\Omega\left(p\frac{|\nabla u|^{p-2}}{\lambda^{p-2}}+q a(x)\frac{|\nabla u|^{q-2}}{\lambda^{q-2}}\right)
\frac{\nabla u}{\lambda}\cdot\nabla \varphi \,dx\geq 0\quad\text{for all $\varphi\in W^{1,\mathcal H}_0(\Omega)$, $\varphi\geq 0$,}
\end{equation}
namely $u$ is a nonnegative supersolution for the equation
\begin{equation}
	\label{parteP}
-\mathrm{div}\Big(p\left|\frac{\nabla u}{\lambda}\right|^{p-2}\frac{\nabla u}{\lambda}+qa(x)\left|\frac{\nabla u}{\lambda}\right|^{q-2}\frac{\nabla u}{\lambda}\Big)=0.
\end{equation}
Considering now radii $r_2>r_1>0$, $k\geq 0$, a cut-off $\eta$ with  $\eta=1$ on $B_{r_1}$ and $\eta=0$ on $B_{r_2}^c$ and
choosing the bounded test function $\varphi :=(u/\lambda-k)_-\eta^q$, if $h(x,t):=t^{p-1}+a(x)t^{q-1}$, recalling that $p<q$ it holds
 \begin{equation}
 \label{first-ma}
 q^2\int_{B_{r_2}}h\Big(x,\frac{|\nabla u|}\lambda\Big)\Big(\frac u\lambda-k\Big)_-|\nabla \eta|\eta^{q-1}\,dx\geq p\int_{B_{r_2}}\mathcal H\Big(x,\frac{|\nabla u|}\lambda\Big)\eta^q\chi_{\{u/\lambda\leq k\}}\,dx.
 \end{equation}
Young's inequality with exponents $(p',p)$ yields for some $c_1(p,q)>0$,
 \begin{align*}
 q^2\int_{B_{r_2}}&\Big(\frac{|\nabla u|}\lambda\Big)^{p-1}\Big(\frac u\lambda-k\Big)_-|\nabla\eta|\eta^{q-1}\,dx \\
 &\leq \frac p2\int_{B_{r_2}}\Big(\frac{|\nabla u|}\lambda\Big)^p\eta^{p'(q-1)}\chi_{\{u/\lambda\leq k\}}\,dx +c_1(p,q)\int_{B_{r_2}}\Big(\frac u\lambda-k\Big)^p_-|\nabla \eta|^p\,dx \\
  &\leq \frac p2\int_{B_{r_2}}\Big(\frac{|\nabla u|}\lambda\Big)^p\eta^{q}\chi_{\{u/\lambda\leq k\}}\,dx +c_1(p,q)\int_{B_{r_2}}
  \left(\Big(\frac u\lambda-k\Big)_- \|\nabla \eta\|_{L^\infty}\right)^p\,dx, 
 \end{align*}
 where we have used that $p'(q-1)\geq q$, since $q>p$ by assumption. 
 Analogously, using Young's inequality with exponents $(q',q)$ also yields, for some $c_2(p,q)>0$,
 \begin{align*}
 q^2\int_{B_{r_2}}& a(x)\Big(\frac{|\nabla u|}\lambda\Big)^{q-1}\Big(\frac u\lambda-k\Big)_-|\nabla \eta|\eta^{q-1}\,dx \\
 &\leq \frac p2\int_{B_{r_2}}a(x)\Big(\frac{|\nabla u|}\lambda\Big)^q\eta^{q}\chi_{\{u/\lambda\leq k\}}\,dx +c_2(p,q)\int_{B_{r_2}}a(x)
 \left(\Big(\frac u\lambda-k\Big)_- \|\nabla\eta\|_{L^\infty}\right)^q\,dx.
 \end{align*}
Whence, by absorbing the first two terms of the right-hand sides into \eqref{first-ma}, we conclude that
 \[
 \int_{B_{r_1}}\mathcal H\Big(x,\frac{|\nabla u|}\lambda\Big)\chi_{\{u/\lambda\leq k\}}\,dx\leq c(p,q)\int_{B_{r_2}}\mathcal H\Big( x,\Big(\frac u\lambda-k\Big)_-\frac{1}{(r_2-r_1)}\Big)\,dx,
 \]
for some positive constant $c(p,q)$,
since $\|\nabla \eta\|_{L^\infty}$ is of order $(r_2-r_1)^{-1}$.
It follows that the function $u/\lambda$ belongs to the De Giorgi class $DG_-$ (cf.\ Section 6 of \cite{BarColMin1}).
Then, in turn, by a slight modification of Theorem 3.5 of \cite{BarColMin1} (in order to allow supersolutions of 
equation \eqref{parteP}), $u$ satisfies the weak Harnack
inequality, yielding
$$
\inf_{B_r(x)} u\geq \frac{1}{c}\Big(\dashint_{B_{2r}(x)} u^\eps(y) dy\Big)^{1/\eps},
$$
for some constants $c\geq 1$ and $\eps\in (0,1)$ and for 
every ball $B_{2r}(x)\subset\Omega$. This immediately yields $u>0$, by a standard argument.
The boundedness of eigenfunctions follows by 
subsection \ref{boundd-L}. Finally, the stability and symmetry properties 
follow by Theorems~\ref{stability},~\ref{faberK} and \ref{partial}.
\hfill$\Box$

%


\begin{remark}\label{ext}{\rm If $\Omega\subset\tilde{\Omega}$, then for all $u\in W^{1,\mathcal H}_0(\Omega)$, the extension by zero 
$$\tilde{u}:=\begin{cases}u\quad&\mbox{in }\Omega,\\
0\quad&\mbox{in }\tilde{\Omega}\setminus\Omega\end{cases}$$ 
belongs to $W^{1,\tilde{\mathcal H}}_0(\tilde{\Omega})$, where $\tilde{\mathcal H}(x,t):=t^p+\tilde{a}(x)t^q$ for all $(x,t)\in\tilde{\Omega}\times[0,\infty)$, and $\tilde{a}$ extends by zero $a$ in $\tilde{\Omega}\setminus\Omega$. Indeed, being $W^{1,\mathcal H}_0(\Omega)\hookrightarrow W^{1,p}_0(\Omega)$, we can use a classical result in $W^{1,p}_0(\Omega)$ (see e.g. Proposition 9.18 of \cite{Brezis}) to obtain that $$\nabla \tilde{u}=\begin{cases}\nabla u\quad&\mbox{in }\Omega,\\
0\quad&\mbox{in }\tilde{\Omega}\setminus\Omega.\end{cases}$$
Since $u\in W^{1,\mathcal H}_0(\Omega)$, there exists a sequence $(\varphi_h)\subset C^\infty_0(\Omega)$ such that $\|\nabla \varphi_h-\nabla u\|_\mathcal H\to 0$ as $h\to\infty$. By Proposition 2.1.11 of \cite{Base}, norm convergence and modular convergence are equivalent, thus $\varrho_\mathcal H(\nabla \varphi_h-\nabla u)\to 0$ as $h\to\infty$.  Extending by zero $a$ and each $\varphi_h$ in $\tilde{\Omega}\setminus \Omega$, we have $(\varphi_h)\subset C^\infty_0(\tilde{\Omega})$, and so    
$$\begin{aligned}\varrho_\mathcal H(\nabla\varphi_h-\nabla u)&=\int_\Omega(|\nabla\varphi_h-\nabla u|^p+a(x)|\nabla\varphi_h-\nabla u|^q)dx\\
&=\int_{\tilde{\Omega}}(|\nabla\varphi_h-\nabla \tilde{u}|^p+\tilde{a}(x)|\nabla\varphi_h-\nabla \tilde{u}|^q)dx=\tilde{\varrho}_{\tilde{\mathcal H}}(\nabla\varphi_h-\nabla \tilde{u}).
\end{aligned}$$
Therefore, $\|\nabla\varphi_h-\nabla \tilde{u}\|_{L^{\tilde{\mathcal H}}(\tilde{\Omega})}\to 0$ as $h\to\infty$ and so $\tilde{u}\in W^{1,\tilde{\mathcal H}}_0(\tilde{\Omega})$.}  
\end{remark}

\begin{theorem}[Stability in domains] 
	\label{stability}
	Let $(\Omega_h)$ be a strictly increasing sequence of 
	open subsets of $\mathbb R^n$ such that
$$
\Omega=\bigcup_{h=1}^\infty\Omega_h.
$$
Then 
$$\label{vardom}\lim_{h\to\infty}\lambda^1_\mathcal H(\Omega_h)=\lambda^1_\mathcal H,$$
(we omit the dependence when the domain is $\Omega$). 
\end{theorem}
\begin{proof} Extending the functions $u\in W^{1,\mathcal H}_0(\Omega_h)$ as zero in $\Omega\setminus\Omega_h$, by Remark \ref{ext} we get $u\in W^{1,\mathcal H}_0(\Omega)$ and clearly 
\begin{equation}\label{ge}\lambda^1_\mathcal H(\Omega_1)\ge \lambda^1_\mathcal H(\Omega_2)\ge\dots\ge\lambda^1_\mathcal H.\end{equation}
On the other hand, by density, $\lambda^1_\mathcal H=\inf_{u\in C^\infty_0(\Omega)\setminus\{0\}}\|\nabla u\|_\mathcal H/\|u\|_\mathcal H$. So, fixed $\varepsilon>0$, we can find $\varphi\in C^\infty_0(\Omega)$ for which 
\begin{equation}\label{varvar}\lambda^1_\mathcal H>\frac{\|\nabla \varphi\|_\mathcal H}{\|\varphi\|_\mathcal H}-\varepsilon.\end{equation}
Since the support of $\varphi$ is compact, it is covered by a finite number of $\Omega_h$'s, hence for $h$ sufficiently large $\mathrm{supp}\varphi\subset\Omega_h$. Whence, 
$$\lambda^1_\mathcal H(\Omega_h)\le\frac{\|\nabla \varphi\|_{L^\mathcal H(\Omega_h)}}{\|\varphi\|_{L^\mathcal H(\Omega_h)}}=\frac{\|\nabla \varphi\|_\mathcal H}{\|\varphi\|_\mathcal H}$$  
and by \eqref{varvar}
$$\lambda^1_\mathcal H>\lambda^1_\mathcal H(\Omega_h)-\varepsilon\quad\mbox{ for $h$ large.}$$
By the arbitrariness of $\varepsilon$
$$\lambda^1_\mathcal H\ge\lim_{h\to\infty}\lambda^1_\mathcal H(\Omega_h)$$
which, combined with \eqref{ge}, gives the conclusion. 
\end{proof}

\begin{theorem}[Isoperimetric property] 
	\label{faberK}
	Let $a\equiv 1$ and $\Omega^*$ be the ball of $\mathbb R^n$ such that $|\Omega^*|=|\Omega|$. Then we have
\begin{equation}
\label{FK-p}
\lambda^1_\mathcal H(\Omega^*)\le\lambda^1_\mathcal H.
\end{equation} 
Moreover, if equality holds in \eqref{FK-p}, then $\Omega$ is a ball. In other
words, balls uniquely minimize the first eigenvalue among sets with given $n$-dimensional Lebesgue measure.
\end{theorem}
\begin{proof} 
Let us prove \eqref{FK-p}.\
Let $u^*$ be the Schwarz symmetrization of a given 
nonnegative function $u\in W^{1,\mathcal H}_0(\Omega)$,
namely the unique radially symmetric and decreasing function with
$$
|\{x\in\Omega^*: u^*(x)>t\}|=|\{x\in\Omega: u(x)>t\}|,\quad \text{for all $t>0$.}
$$
 Since $a\equiv 1$, by (i) of Proposition~\ref{embeddingW}, we have
 $$
 W^{1,\mathcal H}_0(\Omega)\hookrightarrow W^{1,p}_0(\Omega),\qquad
  W^{1,\mathcal H}_0(\Omega)\hookrightarrow W^{1,q}_0(\Omega).
 $$
Then, in light of the P\'olya-Szeg\"o's inequality, we get  $u^*\in W^{1,p}_0(\Omega^*)\cap W^{1,q}_0(\Omega^*)$ and
$$
\varrho^*_\mathcal H(\nabla u^*)=\int_{\Omega^*}(|\nabla u^*|^p+|\nabla u^*|^q)dx\le\int_\Omega(|\nabla u|^p+|\nabla u|^q)dx=\varrho_\mathcal H(\nabla u),
$$
so $u^*\in W^{1,\mathcal H}_0(\Omega^*)$. For $u\in W^{1,\mathcal H}_0(\Omega)\setminus\{0\}$,
$$\begin{aligned}\varrho^*_\mathcal H\left(\frac{\nabla u^*}{\|\nabla u\|_\mathcal H}\right)&=\frac1{\|\nabla u\|^p_\mathcal H}\int_{\Omega^*}|\nabla u^*|^p dx + \frac1{\|\nabla u\|^q_\mathcal H}\int_{\Omega^*}|\nabla u^*|^q dx\\
& \le \frac1{\|\nabla u\|^p_\mathcal H}\int_{\Omega}|\nabla u|^p dx + \frac1{\|\nabla u\|^q_\mathcal H}\int_{\Omega}|\nabla u|^q dx=\varrho_\mathcal H\left(\frac{\nabla u}{\|\nabla u\|_\mathcal H}\right)=1
\end{aligned}$$
which gives, by the unit ball property,
$$
\|\nabla u^*\|_{L^\mathcal H(\Omega^*)}\le\|\nabla u\|_\mathcal H.
$$
On the other hand, since Schwarz symmetrization preserves all $L^p$-norms,
$$\varrho^*_\mathcal H(u^*)=\int_{\Omega^*} (|u^*|^p+|u^*|^q)dx=\int_\Omega(|u|^p+|u|^q)dx=\varrho_\mathcal H(u).$$
Thus, again the unit ball property gives $$\|u^*\|_{L^\mathcal H(\Omega^*)}=\|u\|_\mathcal H.$$
Hence, if we take $u=u^1_{\mathcal H}\geq 0$ a.e., we obtain 
$$
\lambda^1_\mathcal H(\Omega^*)\le\frac{\|\nabla (u^1_\mathcal H)^*\|_{L^\mathcal H(\Omega^*)}}{\|(u^1_\mathcal H)^*\|_{L^\mathcal H(\Omega^*)}}\le\frac{\|\nabla u^1_\mathcal H\|_\mathcal H}{\|u^1_\mathcal H\|_\mathcal H}=\lambda^1_\mathcal H,
$$
which concludes the proof. Assume now that equality holds in inequality \eqref{FK-p}
and consider a first nonnegative eigenfunction $w$ for $\lambda^1_{\mathcal H}$.
Then, recalling that $\|w^*\|_{L^\mathcal H(\Omega^*)}=
\|w\|_{\mathcal H}
$, we conclude by the very definition of $\lambda^1_{\mathcal H}$ that
$\|\nabla w^*\|_{L^\mathcal H(\Omega^*)}=\|\nabla w\|_\mathcal H$.
This, in light of \eqref{norm-mod} gives
\begin{equation}
\label{mod==}
\varrho^*_{\mathcal H}\left(\frac{\nabla w^*}{\|\nabla w\|_{\mathcal H}}\right)=\varrho_{\mathcal H}\left(\frac{\nabla w}{\|\nabla w\|_{\mathcal H}}\right).
\end{equation}
Since, separately,  we have 
$$
\int_{\Omega^*}|\nabla w^*|^pdx\le\int_\Omega |\nabla w|^pdx,\qquad
\int_{\Omega^*}|\nabla w^*|^q dx\le\int_\Omega |\nabla w|^q dx, 
$$
we deduce from identity \eqref{mod==} (in which the denominators agree), that
$$
\|\nabla w^*\|_{L^p(\Omega^*)}=\|\nabla w\|_{L^p(\Omega)}.
$$
This implies (see e.g. \cite{FuscoMaggiPratelli})
that the superlevels of $w$ are balls and, thus, 
$\Omega$ is a ball, completing the proof.
\end{proof}

\begin{remark}\rm
Theorem~\ref{faberK} represents an extension to the double phase case of the so called Faber-Krahn inequality (cf.\
\cite{FuscoMaggiPratelli} for the single phase case).
It was firstly shown by Faber and Krahn \cite{faber,krahn,krahn2} that the first eigenvalue of 
$-\Delta$ on a bounded open set of $\R^2$ of given area attains its minimum value if and only
if is a disk, namely the gravest principal tone is obtained
in the case of a circular membrane, as conjectured by Lord Rayleigh in 1877 \cite{lord}.
\end{remark}


\noindent
A subset $H$ of $\R^N$ is called a {\it polarizer} if it is a closed affine half-space
of $\R^N$, namely the set of points $x$ which satisfy $\alpha\cdot x\leq \beta$
for some $\alpha\in \R^N$ and $\beta\in\R$ with $|\alpha|=1$. Given $x$ in $\R^N$
and a polarizer $H$ the reflection of $x$ with respect to the boundary of $H$ is
denoted by $x_H$. The polarization of a function $u:\R^N\to\R^+$ by a polarizer $H$
is the function $u^H:\R^N\to\R^+$ defined by
\begin{equation}
 \label{polarizationdef}
u^H(x)=
\begin{cases}
 \max\{u(x),u(x_H)\}, & \text{if $x\in H$} \\
 \min\{u(x),u(x_H)\}, & \text{if $x\in \R^N\setminus H$.} \\
\end{cases}
\end{equation}
The {\em polarization} $C^H\subset\R^N$ of a set $C\subset\R^N$ is 
defined as the unique set which satisfies $\chi_{C^H}=(\chi_C)^H$,
where $\chi$ denotes the characteristic function. This operation 
should not be confused with $C_H$ which denotes the {\em reflection} of $C$
with respect to $\partial H$.
The polarization $u^H$ of a positive function $u$ defined on $C\subset \R^N$
is the restriction to $C^H$ of the polarization of the extension $\tilde u:\R^N\to\R^+$ of 
$u$ by zero outside $C$. The polarization of a function which may change sign is defined
by $u^H:=|u|^H$, for any given polarizer $H$.

\begin{theorem}[Partial simmetries] 
	\label{partial}
Let $a\equiv 1$, $H\subset\R^N$ be a half-space and assume that 
$\Omega=\Omega^H$. Then there exists a nonnegative first eigenfunction $u\in W^{1,\mathcal H}_0(\Omega)$ such that $u=u^H$.
\end{theorem}
\begin{proof} 
Let $u\in W^{1,\mathcal H}_0(\Omega)$ with $u\geq 0$ a.e.\ be given. Then,
since $u\in W^{1,p}_0(\Omega)$ and $u\in W^{1,q}_0(\Omega)$, 
by Proposition 2.3 of \cite{vs}, we have
$$
\varrho_\mathcal H(\nabla u^H)=\int_{\Omega}(|\nabla u^H|^p+|\nabla u^H|^q)dx=\int_\Omega(|\nabla u|^p+|\nabla u|^q)dx=\varrho_\mathcal H(\nabla u),$$
so $u^H\in W^{1,\mathcal H}_0(\Omega)$ and  by the unit ball property
$$
\|\nabla u^H\|_{\mathcal H}=\|\nabla u\|_\mathcal H.
$$
Analogously, we have  $\| u^H\|_{\mathcal H}=\|u\|_\mathcal H.$
We want to apply the symmetric Ekeland Variational 
Principle with constraint (see Section 2.4, p. 334 of \cite{sq}, see also \cite{sq-ekeland}) by choosing
$$
X:=W^{1,\mathcal H}_0(\Omega),\quad
S:=W^{1,\mathcal H}_0(\Omega,\R^+),
\quad V:=L^p(\Omega),
\quad f(u):=\|\nabla u\|_\mathcal H, \,\,\,\, u\in  X.
$$
Let $(v_h)\subset \mathcal M$ be a nonnegative
minimization sequence, namely
$$
\lim_{h\to\infty}\|\nabla v_h\|_\mathcal H=\lambda_\mathcal H^{1}.
$$
Then, there exists a new 
minimization sequence $(\tilde v_h)\subset \mathcal M$
such that
\begin{equation}
\label{symmetr}
\||\tilde v_h|^H-\tilde v_h\|_{p}\to 0,\quad\text{as $h\to\infty$}.
\end{equation}
Up to a subsequence $(\tilde v_{h_j})$ converges weakly to $u$ in $W^{1,\mathcal H}_0(\Omega)$ and, in light of Proposition~\ref{sob-poinc}-$(iii)$, we obtain $\|\tilde v_{h_j}-u\|_\mathcal H\to 0$ as $j\to\infty$ (and hence
$\|\tilde v_{h_j}-u\|_p\to 0$ as $j\to\infty$) so that $\|u\|_\mathcal H=1$. This easily implies that $u$ is a minimizer of \eqref{R}. Finally, observing that
(standard contractivity of the polarization in the $L^p$-norm)
$$
\||\tilde v_{h_j}|^H-|u|^H\|_{p}\leq 
\||\tilde v_{h_j}|-|u|\|_{p}\leq \|\tilde v_{h_j}-u\|_{p}\to 0,\quad\text{as $h\to\infty$},
$$
which, taking into account \eqref{symmetr}, yields
$$
\||u|^H-u\|_{p}\leq \||\tilde v_{h_j}|^H-|u|^H\|_{p}+
\||\tilde v_{h_j}|^H-\tilde v_{h_j}\|_{p}+\|\tilde v_{h_j}-u\|_{p}\to 0,\quad\text{as $h\to\infty$},
$$
which yields $|u|^H=u$. Hence $u\geq 0$ and $u^H=u$, concluding the proof.
\end{proof}

\subsection{Large exponents}
\noindent The next result concerns the behavior of the first eigenvalue $\lambda^1_\mathcal H$, when the exponents $p$ and $q$ of the $N$-function $\mathcal H$ are replaced by $hp$ and $hq$, respectively, and $h$ goes to infinity. The passage to infinity was first studied in \cite{JLM} for the $p$-Laplacian operator and then in \cite{Franzi} for the $p(x)$-Laplacian. 

\noindent Clearly, in order to study the $\infty$-eigenvalue problem, we do not require any bound from above on the exponents $hq$, we only assume that $1<p<q$.
Furthermore, throughout this subsection, we use the rescaled modular 
$$\tilde{\varrho}_\mathcal H(u):=\frac1{|\Omega|+\|a\|_1}\int_\Omega (|u|^p+a(x)|u|^q)dx$$
and we denote by $\||\cdot\||_\mathcal H$ the corresponding norm. It is easy to see that $\||\cdot\||_\mathcal H$ is equivalent to $\|\cdot\|_\mathcal H$, more precisely, by (2.1.5) of \cite{Base} and by the unit ball property,
\begin{equation}\label{equivalence}\begin{aligned}
\||u\||_\mathcal H\le \|u\|_\mathcal H\le (|\Omega|+\|a\|_1)\||u\||_\mathcal H, \quad&\mbox{if }|\Omega|+\|a\|_1\ge1,\\
(|\Omega|+\|a\|_1)\||u\||_\mathcal H\le\|u\|_\mathcal H\le \||u\||_\mathcal H, \quad&\mbox{if }|\Omega|+\|a\|_1<1.
\end{aligned}\end{equation}
 
\smallskip

\noindent
We introduce the distance function 
$$
\delta(x):=\mathrm{dist}(x,\partial\Omega),\quad \text{for all $x\in\Omega$}. 
$$
We recall that $\delta$ is Lipschitz continuous and that $\nabla \delta=1$ a.e. in $\Omega$. We define
\begin{equation}\label{lambdainfty}\lambda^1_\infty:=\inf_{u\in W^{1,\infty}_0(\Omega)\setminus\{0\}}\frac{\|\nabla u\|_\infty}{\|u\|_\infty}\end{equation}
Proceeding as in Section 4 of \cite{Franzi} it is easy to see that the minimum in \eqref{lambdainfty} is reached on the distance function and so 
$$\lambda^1_\infty=\frac1{\|\delta\|_\infty}=\frac1R,$$
where $R$ is the so-called {\it inradius}, i.e. the radius of the largest ball inscribed in $\Omega$. 

\noindent
For all $h\in\mathbb N$, put 
$$
(h\mathcal H)(x,t):=t^{hp}+a(x)t^{hq},\quad\text{for all $(x,t)\in\Omega\times[0,\infty)$}.
$$

\begin{lemma}\label{inftynorm} Let $u\in L^\infty(\Omega)$ then 
$$\lim_{h\to\infty}\||u\||_{h\mathcal H}=\|u\|_\infty.$$
\end{lemma}
\begin{proof} First, we want to show that 
\begin{equation}\label{limsupinfty}\limsup_{h\to\infty}\||u\||_{h\mathcal H}\le\|u\|_\infty.\end{equation} 
To this aim, it is enough to consider only those indices $h$ for which $\||u\||_{h\mathcal H}>\|u\|_\infty$,
$$\begin{aligned}1&=\left[\tilde{\varrho}_{h\mathcal H}\left(\frac{u}{\||u\||_{h\mathcal H}}\right)\right]^{\frac1{hp}}=\left[\int_\Omega\left(\left|\frac{u}{\||u\||_{h\mathcal H}}\right|^{hp}+a(x)\left|\frac{u}{\||u\||_{h\mathcal H}}\right|^{hq}\right)\frac1{|\Omega|+\|a\|_1}dx\right]^{\frac1{hp}}\\
&\le \left[\int_\Omega \left(\frac{\|u\|_\infty}{\||u\||_{h\mathcal H}}\right)^{hp}\frac{1+a(x)}{|\Omega|+\|a\|_1}dx\right]^{\frac1{hp}} = \frac{\|u\|_\infty}{\||u\||_{h\mathcal H}}. 
\end{aligned}$$
This implies \eqref{limsupinfty}. Now, in order to prove 
$$\liminf_{h\to\infty}\||u\||_{h\mathcal H}\ge\|u\|_\infty,$$
we assume that $\|u\|_\infty>0$ (the other case is obvious). Then, given $\varepsilon>0$, we can find a set $A_\varepsilon\subset\Omega$, with $|A_\varepsilon|>0$, such that $|u(x)|>\|u\|_\infty-\varepsilon$ for all $x\in A_\varepsilon$. We consider only those indices $h$ for which $\||u\||_{h\mathcal H}\ge\|u\|_\infty-\varepsilon$ and we have 
$$\begin{aligned}1&=\left[\tilde{\varrho}_{h\mathcal H}\left(\frac{u}{\||u\||_{h\mathcal H}}\right)\right]^{\frac1{hp}}\ge\left[\int_{A_\varepsilon}\left(\left|\frac{u}{\||u\||_{h\mathcal H}}\right|^{hp}+a(x)\left|\frac{u}{\||u\||_{h\mathcal H}}\right|^{hq}\right)\frac1{|\Omega|+\|a\|_1}dx\right]^{\frac1{hp}}\\
&>\left[\int_{A_\varepsilon}\left(\frac{\|u\|_\infty-\varepsilon}{\||u\||_{h\mathcal H}}\right)^{hp}\frac1{|\Omega|+\|a\|_1}dx\right]^{\frac1{hp}}=\left(\frac{|A_\varepsilon|}{|\Omega|+\|a\|_1}\right)^{\frac1{hp}}\frac{\|u\|_\infty-\varepsilon}{\||u\||_{h\mathcal H}},\end{aligned}$$
which gives $\liminf_{h\to\infty}\||u\||_{h\mathcal H}\ge\|u\|_\infty-\varepsilon$ and by the arbitrariness of $\varepsilon$ we conclude.
\end{proof}

\noindent We remark that the same property stated in Lemma \ref{inftynorm} holds if we endow the space $L^\mathcal H(\Omega)$ with the standard modular $\varrho_\mathcal H$ and the corresponding norm $\|\cdot\|_\mathcal H$. Furthermore, if we consider the norm in $L^r(\Omega)$ $$\||u\||_r:= \frac{\|u\|_r}{(|\Omega|+\|a\|_1)^{1/r}},$$
the classical result $\lim_{r\to\infty}\||u\||_r=\|u\|_\infty$ continues to hold.  

\begin{theorem}\label{largeeigen} 
	There holds
	$$
	\lim_{h\to\infty}\tilde{\lambda}^1_{h\mathcal H}=\lambda^1_\infty,
	$$
	where $$\tilde{\lambda}^1_{h\mathcal H}:=\inf_{u\in W^{1,\mathcal H}_0(\Omega)\setminus\{0\}}\frac{\||\nabla u\||_\mathcal H}{\||u\||_\mathcal H}.$$
\end{theorem}
\begin{proof} 
	By the definition of $\tilde{\lambda}^1_{h\mathcal H}$, using $\delta$ as test function, we have
	$$\tilde{\lambda}^1_{h\mathcal H}\le\frac{\||\nabla \delta\||_{h\mathcal H}}{\||\delta\||_{h\mathcal H}}\quad\mbox{for all }h\in\mathbb N.$$
	Thus, passing to the limit superior and taking into account Lemma \ref{inftynorm}, we obtain
	\begin{equation}\label{suphmathcalH}\limsup_{h\to\infty}\tilde{\lambda}^1_{h\mathcal H}\le\frac{\|\nabla \delta\|_\infty}{\|\delta\|_\infty}=\lambda^1_\infty.\end{equation}
	Let $(u_h)$ be the sequence of first eigenfunctions corresponding to $\tilde{\lambda}^1_{h\mathcal H}$, with $\||u_h\||_{h\mathcal H}=1$ for all $h$.
	Pick any subsequence $(u_{h_j})$ of $(u_{h})$. Then, $\tilde{\lambda}^1_{h_j\mathcal H}=\||\nabla u_{h_j}\||_{h_j\mathcal H}$ and, by \eqref{suphmathcalH}, the sequence $(\||\nabla u_{h_j}\||_{h\mathcal H})$ is bounded. Therefore, in correspondence to any $r\in[1,\infty)$ there is an integer $j_r$ such that $h_jp\ge r$ for all $j\ge j_r$, and consequently $$W^{1,h_j\mathcal H}_0(\Omega)\hookrightarrow W^{1,r}_0(\Omega)\hookrightarrow\hookrightarrow L^r(\Omega)\qquad\mbox{for all }j\ge j_r.$$ Hence, $(u_{h_j})$ is definitely bounded in the reflexive Banach space $W^{1,r}_0(\Omega)$ and we can extract a subsequence, still denoted by $(u_{h_j})$, for which 
	$$\nabla u_{h_j}\rightharpoonup\nabla u_\infty\quad\mbox{and}\quad u_{h_j}\to u_\infty\quad\mbox{in }L^r(\Omega).$$
	By the arbitrariness of $r$ and the fact that $\Omega$ is bounded, we get $u_\infty\in W^{1,\infty}_0(\Omega)$. In particular, $u_{h_j}\to u_\infty$ also in $L^\infty(\Omega)$, since, by Proposition \ref{embeddingW}-$(ii)$, $u_{h_j}\in L^\infty(\Omega)$ for $j$ large.
	By H\"older's inequality, for all $u\in L^r(\Omega)$
	$$
	\tilde{\varrho}_r(u)=\int_\Omega \frac{|u|^r}{|\Omega|+\|a\|_1}dx\le\left(\int_\Omega\frac{|u|^{h_jp}}{|\Omega|+\|a\|_1}dx\right)^{\frac{r}{h_jp}}\left(\int_\Omega\frac{1+a(x)}{|\Omega|+\|a\|_1}dx\right)^{\frac{h_jp-r}{h_jp}}\le (\tilde{\varrho}_{h_j\mathcal H}(u))^{\frac{r}{h_jp}},
	$$
	whence
	\begin{equation}\label{holdrhp}\||u\||_r\le\||u\||_{h_j\mathcal H}\quad\mbox{for all }j\ge j_r,\end{equation}
	by the unit ball property. 
	We know that $u_{h_j}\in L^\infty(\Omega)$ for $j$ large, so 
	$$\tilde{\varrho}_{h_j\mathcal H}\left(\frac{u_{h_j}}{\|u_{h_j}\|_\infty}\right)\le\int_\Omega\left(\left\|\frac{u_{h_j}}{\|u_{h_j}\|_\infty}\right\|_\infty^{h_jp}+a(x)\left\|\frac{u_{h_j}}{\|u_{h_j}\|_\infty}\right\|_\infty^{h_jq}\right)\frac1{|\Omega|+\|a\|_1}dx=1.$$
	By the unit ball property, we obtain
	\begin{equation}\label{uff}1=\||u_{h_j}\||_{h_j\mathcal H}\le\|u_{h_j}\|_\infty.\end{equation}
	By the weak lower semicontinuity of the $W^{1,r}_0(\Omega)$-norm and by virtue of \eqref{holdrhp} and \eqref{uff},
	$$\begin{aligned}\frac{\||\nabla u_\infty\||_r}{\||u_\infty\||_r}&\le\liminf_{j\to\infty}\frac{\||\nabla u_{h_j}\||_r}{\||u_{h_j}\||_r}\le\liminf_{j\to\infty}\frac{\||\nabla u_{h_j}\||_{h_j\mathcal H}}{\||u_{h_j}\||_r}\\
	&\le\liminf_{j\to\infty}\left(\||\nabla u_{h_j}\||_{h_j\mathcal H}\frac{\|u_{h_j}\|_\infty}{\||u_{h_j}\||_r}\right)=\frac{\|u_\infty\|_\infty}{\||u_\infty\||_r}\liminf_{j\to\infty}\tilde{\lambda}^1_{h_j\mathcal H}\end{aligned}$$
	Hence, 
	$$\lambda^1_\infty\le\frac{\|\nabla u_\infty\|_\infty}{\|u_\infty\|_\infty}=\lim_{r\to\infty}\frac{\||\nabla u_\infty\||_r}{\||u_\infty\||_r}\le\lim_{r\to\infty}\frac{\|u_\infty\|_\infty}{\||u_\infty\||_r}\liminf_{j\to\infty}\tilde{\lambda}^1_{h_j\mathcal H}=\liminf_{j\to\infty}\tilde{\lambda}^1_{h_j\mathcal H},$$
	which, combined with \eqref{suphmathcalH}, gives 
	$$\lambda^1_\infty=\lim_{j\to\infty}\tilde{\lambda}^1_{h_j\mathcal H}.$$
	Finally, we conclude the proof by the arbitrariness of the subsequence. 
\end{proof}

\begin{remark}{\rm By \eqref{equivalence}, in terms of the first eigenvalues $\lambda^1_{h\mathcal H}$, Theorem \ref{largeeigen} gives
$$
\begin{aligned}&\frac{1}{|\Omega|+\|a\|_1}\lambda^1_\infty\le \liminf_{h\to\infty}\lambda^1_{h\mathcal H}\le \limsup_{h\to\infty}\lambda^1_{h\mathcal H}\le(|\Omega|+\|a\|_1)\lambda^1_\infty, \quad\mbox{if }|\Omega|+\|a\|_1\ge1,\\
&(|\Omega|+\|a\|_1)\lambda^1_\infty\le\liminf_{h\to\infty}\lambda^1_{h\mathcal H}\le\limsup_{h\to\infty}\lambda^1_{h\mathcal H}\le\frac1{|\Omega|+\|a\|_1}\lambda^1_\infty, \quad\mbox{if }|\Omega|+\|a\|_1<1.
\end{aligned}
$$
}
\end{remark}
\subsection{Closedness of the spectrum}

\begin{theorem}[Closedness of $\Lambda$]\label{closed} 
	Assume that \eqref{cond-pq} holds, then the spectrum is a closed set.
\end{theorem}
\begin{proof} Let $(\lambda_h)\subset\Lambda$ be a sequence of eigenvalues of \eqref{EL} 
	converging to a certain $\lambda<\infty$. 
	Let us denote by $(u_h)$ the sequence of the corresponding eigenfunctions such that $\|u_h\|_\mathcal H=1$ for all $h$. Then,
	we have 
	\begin{equation}
	\label{DI}
	\begin{aligned}\int_\Omega&\left[p\left(\frac{|\nabla u_h|}{\lambda_h}\right)^{p-2}+qa(x)\left(\frac{|\nabla u_h|}{\lambda_h}\right)^{q-2}\right]\frac{\nabla u_h}{\lambda_h}\cdot\nabla v \,dx\\
	&\,=\lambda_hS(u_h)\int_\Omega\left(p|u_h|^{p-2}+qa(x)|u_h|^{q-2}\right)u_hvdx\quad\mbox{for all }v\in W^{1,\mathcal H}_0(\Omega)\mbox{ and }h\ge1,
	\end{aligned}
	\end{equation}
	and, by normalization, $\lambda_h=\|\nabla u_h\|_\mathcal H$. Therefore 
	$(u_h)$ is bounded in the reflexive Banach space $W^{1,\mathcal H}_0(\Omega)$ and it
	admits a subsequence $(u_{h_j})$ such that $u_{h_j}\rightharpoonup u$ in $W^{1,\mathcal H}_0(\Omega)$ as $j\to\infty$. Thus, by Proposition~\ref{sob-poinc}-$(iii)$, $u_{h_j}\to u$ in $L^\mathcal H(\Omega)$ strongly as $j\to\infty$, yielding $\|u\|_\mathcal H=1$ 
	(which, in particular, provides $u\neq 0$). 
	We claim that $u$ is an eigenfunction with corresponding eigenvalue $\lambda$, i.e. that the following distributional 
	identity is satisfied for all $v\in C^\infty_0(\Omega)$,
	$$
	\int_\Omega\left[p\left(\frac{|\nabla u|}{\lambda}\right)^{p-2}\!\!+qa(x)\left(\frac{|\nabla u|}{\lambda}\right)^{q-2}\right]\frac{\nabla u}{\lambda}\cdot\nabla v dx=\lambda S(u)\!\!\int_\Omega\!\!\left(p|u|^{p-2}+qa(x)|u|^{q-2}\right)uvdx.
	$$
	In order to prove that, we shall pass to the limit in \eqref{DI}. First, being $L^\mathcal H(\Omega)\hookrightarrow L^p(\Omega)$, by the dominated convergence theorem,
	\begin{equation}
	\label{pezzop}
	\int_\Omega |u_{h_j}|^{p-2}u_{h_j}v\,dx\to \int_\Omega |u|^{p-2}u v\,dx
	\quad\mbox{for all }v\in C^\infty_0(\Omega)
	\end{equation}
	up to a subsequence. Moreover, being $L^\mathcal H(\Omega)\hookrightarrow L^q_a(\Omega)$, $a^{1/q}u_{h_j}\to a^{1/q}u$ in $L^q(\Omega)$ and so, up to a subsequence, there exists $\omega\in L^q(\Omega)$ such that $|a^{1/q}u_{h_j}|\le \omega$ for all $j$. Therefore, $$a|u_{h_j}|^{q-1}|v|=a^{(q-1)/q}|u_{h_j}|^{q-1}a^{1/q}|v|\le\omega^{q-1}a^{1/q}|v|\in L^1(\Omega)\quad\mbox{for all }j,$$
	and so, the dominated convergence theorem implies
	\begin{equation}\label{pezzoq}\int_\Omega a(x)|u_{h_j}|^{q-2}u_{h_j}vdx\to\int_\Omega a(x)|u|^{q-2}uvdx\quad\mbox{for all }v\in C^\infty_0(\Omega).\end{equation}
	Therefore, by \eqref{pezzop}-\eqref{pezzoq}, there exists a subsequence, still denoted by $(u_{h_j})$, for which 
	the following limit holds, for all $v\in C^\infty_0(\Omega)$,
	\begin{equation}
	\label{RHS}
	\lim_{j\to\infty}\int_\Omega\left(p|u_{h_j}|^{p-2}+qa(x)|u_{h_j}|^{q-2}\right)u_{h_j}vdx=\int_\Omega\left(p|u|^{p-2}+qa(x)|u|^{q-2}\right)uvdx.
	\end{equation}
	Now, by \eqref{DI} with $v=\dfrac{u_{h_j}}{\lambda_{h_j}}-\dfrac{u}{\lambda}$, we get 
	\begin{equation}
	\label{DIspecial}
	\begin{aligned}&\int_\Omega\left[\left(p\left|\frac{\nabla u_{h_j}}{\lambda_{h_j}}\right|^{p-2}+qa(x)\left|\frac{\nabla u_{h_j}}{\lambda_{h_j}}\right|^{q-2}\right)\frac{\nabla u_{h_j}}{\lambda_{h_j}}\right.\\
	&\qquad-\left.\left(p\left|\frac{\nabla u}{\lambda}\right|^{p-2}+qa(x)\left|\frac{\nabla u}{\lambda}\right|^{q-2}\right)\frac{\nabla u}\lambda\right]\cdot\nabla\left(\frac{u_{h_j}}{\lambda_{h_j}}-\frac{u}{\lambda}\right)dx\\
	&=\lambda_{h_j}S(u_{h_j})\int_\Omega(p|u_{h_j}|^{p-2}+qa(x)|u_{h_j}|^{q-2})u_{h_j}\left(\frac{u_{h_j}}{\lambda_{h_j}}-\frac{u}{\lambda}\right)dx\\
	&\qquad-\int_\Omega\left(p\left|\frac{\nabla u}{\lambda}\right|^{p-2}+qa(x)\left|\frac{\nabla u}{\lambda}\right|^{q-2}\right)\frac{\nabla u}{\lambda}\cdot\nabla\left(\frac{u_{h_j}}{\lambda_{h_j}}-\frac{u}{\lambda}\right)dx.
	\end{aligned}
	\end{equation}
	Passing to the limit under integral sign (we can reason as for \eqref{pezzop} and \eqref{pezzoq}), and observing that $S(u_{h_j})\le q$ for all $j$,
	$$\lim_{j\to\infty}\lambda_{h_j}S(u_{h_j})\int_\Omega(p|u_{h_j}|^{p-2}+qa(x)|u_{h_j}|^{q-2})u_{h_j}\left(\frac{u_{h_j}}{\lambda_{h_j}}-\frac{u}{\lambda}\right)dx=0.$$
	Furthermore, since $\dfrac{\nabla u_{h_j}}{\lambda_{h_j}}\rightharpoonup \dfrac{\nabla u}{\lambda}$ 
	in $[L^\mathcal H(\Omega)]^n$ and being $L^\mathcal H(\Omega)\hookrightarrow L^p(\Omega)$, we have
	$$
	\lim_{j\to\infty}\int_\Omega\left|\frac{\nabla u}{\lambda}\right|^{p-2}\frac{\nabla u}{\lambda}\cdot\nabla\left(\frac{u_{h_j}}{\lambda_{h_j}}-\frac{\nabla u}{\lambda}\right) dx= 0.
	$$
	Analogously, $L^\mathcal H(\Omega)\hookrightarrow L^q_a(\Omega)$ and for the dual spaces the reverse embedding holds. It is easy to see that the functional $$F:f\in [L^q_a(\Omega)]^n\mapsto\int_\Omega a(x)\left|\frac{\nabla u}{\lambda}\right|^{q-2}\frac{\nabla u}{\lambda}\cdot f dx$$
	belongs to $([L^q_a(\Omega)]^n)'\subset ([L^\mathcal H(\Omega)]^n)'$. Since in particular $\dfrac{\nabla u_{h_j}}{\lambda_{h_j}}\rightharpoonup \dfrac{\nabla u}{\lambda}$ in $[L^q_a(\Omega)]^n$, we get 
	$$\lim_{j\to\infty}\int_\Omega a(x)\left|\frac{\nabla u}{\lambda}\right|^{q-2}\frac{\nabla u}{\lambda}\cdot\nabla\left(\frac{u_{h_j}}{\lambda_{h_j}}-\frac{\nabla u}{\lambda}\right) dx= 0.$$
	Then, by \eqref{DIspecial},
	$$\begin{aligned}\lim_{j\to\infty}&\int_\Omega\left[\left(p\left|\frac{\nabla u_{h_j}}{\lambda_{h_j}}\right|^{p-2}+qa(x)\left|\frac{\nabla u_{h_j}}{\lambda_{h_j}}\right|^{q-2}\right)\frac{\nabla u_{h_j}}{\lambda_{h_j}}\right.\\
	&\qquad-\left.\left(p\left|\frac{\nabla u}{\lambda}\right|^{p-2}+qa(x)\left|\frac{\nabla u}{\lambda}\right|^{q-2}\right)\frac{\nabla u}\lambda\right]\cdot\nabla\left(\frac{u_{h_j}}{\lambda_{h_j}}-\frac{u}{\lambda}\right)dx=0,\end{aligned}$$
	and consequently
	\begin{equation}\label{duelim}\begin{aligned}\lim_{j\to\infty}&\int_\Omega p\left(\left|\frac{\nabla u_{h_j}}{\lambda_{h_j}}\right|^{p-2}\frac{\nabla u_{h_j}}{\lambda_{h_j}}-\left|\frac{\nabla u}{\lambda}\right|^{p-2}\frac{\nabla u}\lambda\right)\cdot\nabla\left(\frac{u_{h_j}}{\lambda_{h_j}}-\frac{u}{\lambda}\right)dx=0,\\
	\lim_{j\to\infty}&\int_\Omega qa(x)\left(\left|\frac{\nabla u_{h_j}}{\lambda_{h_j}}\right|^{q-2}\frac{\nabla u_{h_j}}{\lambda_{h_j}}-\left|\frac{\nabla u}{\lambda}\right|^{q-2}\frac{\nabla u}\lambda\right)\cdot\nabla\left(\frac{u_{h_j}}{\lambda_{h_j}}-\frac{u}{\lambda}\right)dx=0,\end{aligned}\end{equation}
	since, by convexity, the integrands are nonnegative. 
	Now we estimate the term 
	$$
	\mathcal R_q:=\int_\Omega qa(x)\left(\left|\frac{\nabla u_{h_j}}{\lambda_{h_j}}\right|^{q-2}\frac{\nabla u_{h_j}}{\lambda_{h_j}}-\left|\frac{\nabla u}{\lambda}\right|^{q-2}\frac{\nabla u}{\lambda}\right)\cdot\nabla\left(\frac{u_{h_j}}{\lambda_{h_j}}-\frac{u}\lambda\right)dx.
	$$
	We recall that the following inequalities hold for all $s,\,t\in\mathbb R^n$ (cf. inequalities (I) and (VII) of Section 10 of \cite{Lind_disug}),
	\begin{equation}\label{ge2}
	|s-t|^q\le 2^{q-2}(|s|^{q-2}s-|t|^{q-2}t)\cdot(s-t),\quad\mbox{if }q\ge2,
	\end{equation}
	\begin{equation}\label{<2}
	|s-t|^q\le(q-1)^{-\frac{q}2}\left[(|s|^{q-2}s-|t|^{q-2}t)\cdot(s-t)\right]^{\frac{q}{2}}(|s|^2+|t|^2)^{\frac{2-q}2\frac{q}2},\quad\mbox{if }1<q<2.
	\end{equation}
	Therefore, if $q\ge 2$, by \eqref{ge2}, 
	\begin{equation}\label{mag2}\int_\Omega qa(x)\left|\nabla\left(\frac{u_{h_j}}{\lambda_{h_j}}-\frac{u}\lambda\right)\right|^qdx\le 2^{q-2}\mathcal R_q\end{equation}
	If $1<q<2$, H\"older's inequality gives
	$$\begin{aligned}&\int_\Omega \left[qa(x)\left(\left|\frac{\nabla u_{h_j}}{\lambda_{h_j}}\right|^{q-2}\frac{\nabla u_{h_j}}{\lambda_{h_j}}-\left|\frac{\nabla u}{\lambda}\right|^{q-2}\frac{\nabla u}{\lambda}\right)\cdot\nabla\left(\frac{u_{h_j}}{\lambda_{h_j}}-\frac{u}\lambda\right)\right]^{\frac{q}2}\\
	&\hspace{5cm}\cdot\left[(qa(x))^{\frac2q}\left(\left|\frac{\nabla u_{h_j}}{\lambda_{h_j}}\right|^2+\left|\frac{\nabla u}\lambda\right|^2\right)\right]^{\frac{(2-q)q}4}dx\\
	&\le \mathcal R^{\frac{q}2}_q\left[\int_\Omega qa(x)\left(\left|\frac{\nabla u_{h_j}}{\lambda_{h_j}}\right|^2+\left|\frac{\nabla u}\lambda\right|^2\right)^{\frac{q}2}dx\right]^{\frac{2-q}2}\\
	&\le\mathcal R^{\frac{q}2}_q\left[\int_\Omega qa(x)\left(\left|\frac{\nabla u_{h_j}}{\lambda_{h_j}}\right|^q+\left|\frac{\nabla u}\lambda\right|^q\right)dx\right]^{\frac{2-q}2}\le M \mathcal R^{\frac{q}2}_q,
	\end{aligned}$$
	where 
	$M<\infty$ bounds from above the term in square brackets for all $j$, being $(|\nabla u_h|/\lambda_h)$ bounded in $L^\mathcal H(\Omega)$ and so in $L^q_a(\Omega)$. By \eqref{<2}, this implies that 
	\begin{equation}\label{min2}\int_\Omega qa(x)\left|\nabla\left(\frac{u_{h_j}}{\lambda_{h_j}}-\frac{u}\lambda\right)\right|^qdx\le \frac{M}{(q-1)^{\frac{q}2}}\,\mathcal R^{\frac{q}2}_q\end{equation}
	Analogous estimates as \eqref{mag2} and \eqref{min2} hold for the term in $p$, therefore, by \eqref{duelim}, this implies that 
	$$\lim_{j\to\infty}\varrho_\mathcal H\left(\frac{\nabla u_{h_j}}{\lambda_{h_j}}-\frac{\nabla u}{\lambda}\right)\le\lim_{j\to\infty}\int_\Omega\left(p\left|\frac{\nabla u_{h_j}}{\lambda_{h_j}}-\frac{\nabla u}{\lambda}\right|^p+qa(x)\left|\frac{\nabla u_{h_j}}{\lambda_{h_j}}-\frac{\nabla u}{\lambda}\right|^q\right)dx=0.$$
	By Lemma 2.1.11 of \cite{Base}, the norm convergence and the modular convergence are equivalent in the space
	$L^\mathcal H(\Omega)$, so we obtain
	$$
	\dfrac{\nabla u_{h_j}}{\lambda_{h_j}}\to\dfrac{\nabla u}{\lambda},\quad \text{in $[L^\mathcal H(\Omega)]^n$},
	$$ 
	and we can pass to the limit in the left-hand side of \eqref{DI} under the integral sign (as it was already done for \eqref{RHS}) to obtain 
	\begin{equation}
	\label{LHS}
	\begin{aligned}\lim_{j\to\infty}&\int_\Omega\left[p\left(\frac{|\nabla u_{h_j}|}{\lambda_{h_j}}\right)^{p-2}+qa(x)\left(\frac{|\nabla u_{h_j}|}{\lambda_{h_j}}\right)^{q-2}\right]\frac{\nabla u_{h_j}}{\lambda_{h_j}}\cdot\nabla v\, dx\\
	&=\int_\Omega\left[p\left(\frac{|\nabla u|}{\lambda}\right)^{p-2}+qa(x)\left(\frac{|\nabla u|}{\lambda}\right)^{q-2}\right]\frac{\nabla u}{\lambda}\cdot\nabla v\, dx, \quad\mbox{for all }v\in C^\infty_0(\Omega).
	\end{aligned}
	\end{equation}
	Since $u_{h_j}\to u$ in $L^\mathcal H(\Omega)$ and $\dfrac{\nabla u_{h_j}}{\lambda_{h_j}}\to\dfrac{\nabla u}{\lambda}$ in $[L^\mathcal H(\Omega)]^n$, by dominated convergence we also get
	$$
	\lim_{j\to\infty}S(u_{h_j})
	=\lim_{j\to\infty}\frac{\displaystyle\int_\Omega\left[p\left(\frac{|\nabla u_{h_j}|}{\lambda_{h_j}}\right)^p+qa(x)\left(\frac{|\nabla u_{h_j}|}{\lambda_{h_j}}\right)^q\right]dx}{\displaystyle\int_\Omega\left[p|u_{h_j}|^p+qa(x)|u_{h_j}|^q\right]dx}
	=S(u).
	$$ 
	Together with \eqref{RHS} and \eqref{LHS}, this proves claim, concluding the proof.
\end{proof}

\section{Variational eigenvalues}\label{sec4}

\noindent
Throughout this section, we assume that \eqref{cond-pq} holds and, unless explicitly stated, we consider $W^{1,\mathcal H}_0(\Omega)$ equipped with the $L^\mathcal H$-norm of the gradient.

\begin{lemma}\label{dominazione}
Let $u\in L^\mathcal H(\Omega)\setminus\{0\}$. For all $v\in L^\mathcal H(\Omega)$ the following inequality holds
$$|\langle k'(u),v\rangle|\le q\|v\|_\mathcal H.$$
\end{lemma}
\begin{proof} For $v=0$ the thesis is obvious, hence we suppose $v\neq0$ and have by virtue of \eqref{den>1} and by Young's inequality
$$\begin{aligned}|\langle k'(u),v\rangle|&\le\|v\|_\mathcal H\int_\Omega\left[p\left(\frac{|u|}{k(u)}\right)^{p-1}+q a(x)\left(\frac{|u|}{k(u)}\right)^{q-1}\right]\frac{|v|}{k(v)}dx\\
&=\|v\|_\mathcal H\left[\int_\Omega p^{(p-1)/p}\left(\frac{|u|}{k(u)}\right)^{p-1}p^{1/p}\frac{|v|}{k(v)}dx\right.\\
&\phantom{=}+\left.\int_\Omega (qa(x))^{(q-1)/q}\left(\frac{|u|}{k(u)}\right)^{q-1}(qa(x))^{1/q}\frac{|v|}{k(v)}dx\right]\\
&\le\|v\|_\mathcal H\left[\left(1-\frac1p\right)\int_\Omega p\left(\frac{|u|}{k(u)}\right)^p dx+\int_\Omega\left(\frac{|v|}{k(v)}\right)^p dx \right.\\
&\phantom{=}+\left.\left(1-\frac1q\right)\int_\Omega qa(x)\left(\frac{|u|}{k(u)}\right)^q dx+\int_\Omega a(x)\left(\frac{|v|}{k(v)}\right)^q dx\right]\\
&\le\|v\|_\mathcal H\left\{(q-1)\varrho_\mathcal H\left(\frac{u}{k(u)}\right)+\varrho_\mathcal H\left(\frac{v}{k(v)}\right)\right\}=q\|v\|_\mathcal H.\end{aligned}$$ 
This concludes the proof.
\end{proof}

\begin{theorem}\label{PS} $\widetilde{K}:=K\big|_\mathcal M$ satisfies the (PS) condition, i.e. every sequence $(u_h)\subset\mathcal M$ such that $\widetilde{K}(u_h)\to c$ for some $c\in\mathbb R$ and $\widetilde{K}'(u_h)\to 0$ in $(W^{1,\mathcal H}_0(\Omega))'$ admits a convergent subsequence.  
\end{theorem}
\begin{proof} By hypotheses there exist $c\in\mathbb R$ and a sequence $(c_h)\subset\mathbb R$ such that 
\begin{equation}\label{PSseq}K(u_h)\to c\quad\mbox{and}\quad K'(u_h)-c_hk'(u_h)\to 0\mbox{ in }(W^{1,\mathcal H}_0(\Omega))'.\end{equation}
It is easy to see that $\langle K'(u_h),u_h\rangle=K(u_h)$ and $\langle k'(u_h),u_h\rangle=k(u_h)=1$, so \eqref{PSseq} implies that $c_h\to c$. Since $(u_h)$ is bounded in $W^{1,\mathcal H}_0(\Omega)$, up to a subsequence, $u_h\rightharpoonup u$ in $W^{1,\mathcal H}_0(\Omega)$ and $u_h\to u$ in $L^\mathcal H(\Omega)$. Thus, by Lemma \ref{dominazione}, 
$$|\langle k'(u_h),u_h-u\rangle|\le q\|u_h-u\|_\mathcal H\to 0,$$
and so the second limit in \eqref{PSseq} implies
\begin{equation}\label{K'}
\langle K'(u_h),u_h-u\rangle\to0\end{equation}
as $h\to\infty$. Now, by the convexity of the $C^1$ functional $K$, we obtain for all $h$
$$\|\nabla u_h\|_\mathcal H\le\|\nabla u\|_\mathcal H+\langle K'(u_h),u_h-u\rangle.$$
Hence, \eqref{K'} and the weak lower semicontinuity of the norm give
$$\limsup_{h\to\infty}\|\nabla u_h\|_\mathcal H\le\|\nabla u\|_\mathcal H\le\liminf_{h\to\infty}\|\nabla u_h\|_\mathcal H,$$
whence
$$\|\nabla u_h\|_\mathcal H\to\|\nabla u\|_\mathcal H\quad\mbox{as }h\to\infty.$$
Clearly, also $\|u_h\|_\mathcal H\to\|u\|_\mathcal H$. Therefore, being $(W^{1,\mathcal H}_0(\Omega),\|\cdot\|_{1,\mathcal H})$ uniformly convex by Proposition~\ref{reflexivity}, we get $u_h\to u$ in $W^{1,\mathcal H}_0(\Omega)$ and conclude the proof.
\end{proof}

\noindent
The previous result allows us to define a sequence of eigenvalues of \eqref{EL} by a minimax procedure. 

\begin{definition}{\rm For $m\in\mathbb N$, we define the {\it $m$-th variational eigenvalue} $\lambda^m_\mathcal H$ of \eqref{EL} as
\begin{equation}\label{mth-eigen}\lambda^{m}_\mathcal H:=\inf_{K\in\mathcal W^m_\mathcal H}\sup_{u\in K}\|\nabla u\|_\mathcal H,\end{equation}
where $\mathcal W^m_\mathcal H$ is the set of compact subsets $K$ of $\mathcal M:=\{u\in W^{1,\mathcal H}_0(\Omega)\,:\, \|u\|_\mathcal H=1\}$ that are symmetric (i.e. $K=-K$) and have topological index $i(K)\ge m$.}
\end{definition}

\noindent
The topological index $i$ can be chosen as the Krasnosel'ski\u{\i} genus or the $\mathbb Z_2$-cohomological index of Fadell and Rabinowitz, the results below hold with any index $i$ satisfying the following properties: 

\begin{itemize}
\item[($i_1$)]
if $X$ is a topological vector space and
$K\subseteq X\setminus\{0\}$ is compact, symmetric and nonempty,
$i(K)$ is an integer greater or equal than $1$;
\item[($i_2$)]
if $X$ is a topological vector space and
$K\subseteq X\setminus\{0\}$ is compact, symmetric and nonempty,
then there exists an open subset $U$ of $X\setminus\{0\}$
such that $K\subseteq U$ and
$i(\hat{K}) \leq i(K)$
for any compact, symmetric and nonempty $\hat{K}\subseteq U$\,;
\item[($i_3$)]
if $X, Y$ are two topological vector spaces,
$K\subseteq X\setminus\{0\}$ is compact, symmetric and nonempty
and $\pi:K\to Y\setminus\{0\}$ is continuous and
odd, we have $i(\pi(K)) \geq i(K)$\,;
\item[($i_4$)] if $(X,\|\cdot\|)$ is a normed space with $1\le \mathrm {dim} X<\infty$, then 
$$i(\{u\in X\,:\,\|u\|=1\})=\mathrm{dim}X.$$
\end{itemize}

\noindent
We remark that the new definition of $\lambda^1_\mathcal H$ is consistent with \eqref{Rr}, by virtue of property $(i_1)$ and the fact that $K$ is an even functional. 
\medskip

\noindent
$\bullet$ {\it Proof of Theorem \ref{nonlinspec}.} By virtue of Theorem \ref{PS}, we can apply Theorem 5.11 of \cite{Struwe} (see also Propositions 3.52 and 3.53 of \cite{PAORbook}) to prove that the values defined in \eqref{mth-eigen} are actually eigenvalues of \eqref{EL} in the sense of Definition \ref{eigen}. Furthermore, since $\mathcal W^{m+1}_\mathcal H\subseteq\mathcal W^m_\mathcal H$ for all $m$, $(\lambda^{m}_\mathcal H)$ defines a non-decreasing sequence.
Finally, $\lambda^{m}_\mathcal H\nearrow\infty$ as $m\to\infty$, being $i(\mathcal M)=\infty$ by $(i_4)$ and $(i_2)$.\hfill$\Box$

\section{Stability of variational eigenvalues}\label{sec5}
\noindent
In this section, we assume again the validity of condition \eqref{cond-pq} and endow $W^{1,\mathcal H}_0(\Omega)$ with the $L^\mathcal H(\Omega)$-norm of the gradient.

\noindent For any couple of real numbers $(p,q)$ such that $1<p<q< n$, we define the corresponding $N$-function $\mathcal H(x,t):=t^p+a(x)t^q$ for all $(x,t)\in\Omega\times [0,\infty)$ and the related functionals $\mathscr E_\mathcal H:L^1(\Omega)\to[0,\infty]$ 
as
\begin{equation}\label{EH}
\mathscr E_\mathcal H(u):=\begin{cases}\|\nabla u\|_\mathcal H\quad&\mbox{if }u\in W^{1,\mathcal H}_0(\Omega),\\
+\infty\quad&\mbox{otherwise}\end{cases}
\end{equation}  
and $g_\mathcal H:L^1(\Omega)\to[0,\infty)$  as
$$g_\mathcal H(u):=\begin{cases}\|u\|_\mathcal H\quad&\mbox{if }u\in L^\mathcal H(\Omega),\\
0\quad&\mbox{otherwise.}\end{cases}$$

\begin{proposition}\label{g} The following properties hold:
\begin{itemize}
\item[$(i)$] $g_\mathcal H$ is even and positively homogeneous of degree $1$;
\item[$(ii)$] for every $b\in\mathbb R$ the restriction of $g_\mathcal H$ to $\{u\in L^1(\Omega)\,:\, \mathscr E_\mathcal H(u)\le b\}$ is continuous. 
\end{itemize}
\end{proposition}

\noindent For the proof of Proposition \ref{g} we refer the reader to Proposition 2.3 of \cite{cs}, where a similar result is given in the case of functionals defined in variable exponent spaces.\medskip 

\begin{definition}{\rm $((p_h,q_h))\subset\mathbb R^2$ is an {\it admissible non-increasing sequence converging to} $(p,q)$, and we shall write $(p_h,q_h)\searrow (p,q)$, if $q_1< n$, $p_h<q_h$ for all $h\in\mathbb N$, $p_h\searrow p$ and $q_h\searrow q$ as $h\to\infty$.}\end{definition}

\begin{lemma}\label{conver} Let $((p_h,q_h))$ be an admissible non-increasing sequence converging to $(p,q)$. Then, for all $w\in C^1_0(\Omega)$, $$\lim_{h\to\infty}\|\nabla w\|_{\mathcal H_h}=\|\nabla w\|_\mathcal H,$$ where $\mathcal H_h$ is the N-function corresponding to the exponents $p_h$ and $q_h$ for all $h$.
\end{lemma}
\begin{proof} First, by Fatou's lemma we get that 
\begin{equation}\label{lsc}\varrho_\mathcal H(\nabla w)=\int_\Omega(|\nabla w|^p+a(x)|\nabla w|^q)dx \le\liminf_{h\to\infty}\int_\Omega(|\nabla w|^{p_h}+a(x)|\nabla w|^{q_h})dx=\liminf_{h\to\infty}\varrho_{\mathcal H_h}(\nabla w).\end{equation}
Now, let $\alpha:=\liminf_{h\to\infty}\|\nabla w\|_{\mathcal H_h}<\infty$ (in the case $\alpha=\infty$ this part of the proof is obvious) and take $\gamma>\alpha$. There exists a subsequence $((p_{h_j}, q_{h_j}))$ for which $\|\nabla w\|_{\mathcal H_{h_j}}\le \gamma$ for all $j$, and so by the unit ball property and by \eqref{lsc}
$$\varrho_\mathcal H\left(\frac{\nabla w}{\gamma}\right)\le\liminf_{j\to\infty}\varrho_{\mathcal H_{h_j}}\left(\frac{\nabla w}{\gamma}\right)\le 1$$
and so $\|\nabla w\|_\mathcal H\le\gamma$. By the arbitrariness of $\gamma$ we get  
$$\|\nabla w\|_\mathcal H\le\liminf_{h\to\infty}\|\nabla w\|_{\mathcal H_h}.$$
It remains to prove that 
\begin{equation}\label{limsup}\|\nabla w\|_\mathcal H\ge\limsup_{h\to\infty}\|\nabla w\|_{\mathcal H_h}.\end{equation} 
If $\|\nabla w\|_\mathcal H=0$ the conclusion follows immediately. Let us assume that $\|\nabla w\|_\mathcal H>0$ and take any $\gamma\in (0,1)$. By the boundedness of $\Omega$ we get 
$$\left|\frac{\gamma \nabla w}{\|\nabla w\|_\mathcal H}\right|^{p_h}+a(x)\left|\frac{\gamma \nabla w}{\|\nabla w\|_\mathcal H}\right|^{q_h}\le(1+a(x))\left(1+\left|\frac{\gamma \nabla w}{\|\nabla w\|_\mathcal H}\right|^n\right)\in L^1(\Omega).$$
Hence, by the dominated convergence theorem and by (2.1.5) of \cite{Base} 
$$\begin{aligned}\lim_{h\to\infty}\varrho_{\mathcal H_h}\left(\frac{\gamma \nabla w}{\|\nabla w\|_\mathcal H}\right)
&=\int_\Omega\left(\left|\frac{\gamma \nabla w}{\|\nabla w\|_\mathcal H}\right|^{p}+a(x)\left|\frac{\gamma \nabla w}{\|\nabla w\|_\mathcal H}\right|^{q}\right)dx=\varrho_\mathcal H\left(\frac{\gamma \nabla w}{\|\nabla w\|_\mathcal H}\right)\\
&\le \gamma \varrho_\mathcal H\left(\frac{\nabla w}{\|\nabla w\|_\mathcal H}\right)=\gamma<1\end{aligned}$$
Therefore, for $h$ sufficiently large $\varrho_{\mathcal H_h}\left(\gamma \nabla w/\|\nabla w\|_\mathcal H\right)<1$ and by the unit ball property $$\left\|\frac{\gamma \nabla w}{\|\nabla w\|_\mathcal H}\right\|_{\mathcal H_h}<1$$
that is $\|\nabla w\|_{\mathcal H_h}<\|\nabla w\|_\mathcal H/\gamma$. Whence, 
$$\limsup_{h\to\infty}\|\nabla w\|_{\mathcal H_h}\le\frac{\|\nabla w\|_\mathcal H}{\gamma}\quad \mbox{for all }\gamma\in(0,1),$$
which implies \eqref{limsup}.
\end{proof}

\begin{lemma}\label{emb} Put $\tilde{\mathcal H}(x,t):=t^{\tilde{p}}+a(x)t^{\tilde{q}}$ for all $(x,t)\in\Omega\times[0,\infty)$. If $p\le\tilde{p}$ and $q\le\tilde{q}$, with $1<\tilde{p}<\tilde{q}< n$, then $L^{\tilde{\mathcal H}}(\Omega)\hookrightarrow L^{\mathcal H}(\Omega)$, with constant embedding less than or equal to $|\Omega|+\|a\|_1+1$. If furthermore $\tilde{p}< 2p$ and $\tilde{q}< 2q$, the embedding constant is less than or equal to 
\begin{equation}\label{CHH}C_{\tilde{\mathcal H},\mathcal H} := \begin{cases}\dfrac{\tilde{p}-p}p(|\Omega|+\|a\|_1)+\left(\dfrac{\tilde{p}-p}p\right)^{\frac{p-\tilde{p}}p}\quad&\mbox{if }\dfrac{\tilde{q}}q\le\dfrac{\tilde{p}}p,\\
\dfrac{\tilde{q}-q}q(|\Omega|+\|a\|_1)+\left(\dfrac{\tilde{q}-q}q\right)^{\frac{q-\tilde{q}}q}\quad&\mbox{otherwise.}\end{cases}\end{equation}
In particular, if $(p_h,q_h)\searrow (p,q)$, then $C_{\mathcal H_h,\mathcal H}\to 1$ as $h\to\infty$, where $\mathcal H_h(x,t):=t^{p_h}+a(x)t^{q_h}$.
\end{lemma}
\begin{proof} In what follows we shall use Young's inequality in this form, 
\begin{equation}\label{Yi}ab\le\varepsilon a^{p'}+C_\varepsilon b^{p}\quad\mbox{ for all }a,\,b\ge0,\,p>1,\,\varepsilon>0,\end{equation}
with $p'=p/(p-1)$ and $C_\varepsilon=\varepsilon^{-(p-1)}$.
By \eqref{Yi}, for all $(x,t)\in\Omega\times[0,\infty)$ and $\varepsilon>0$
$$
\begin{gathered}
t^p\le\varepsilon+\frac1{\varepsilon^{(\tilde{p}-p)/p}}t^{\tilde{p}}\\
a(x)t^{q}\le \varepsilon a(x)+\frac{a(x)}{\varepsilon^{(\tilde{q}-q)/q}}t^{\tilde{q}}.
\end{gathered}
$$
Thus, 
$$\mathcal H(x,t)\le\varepsilon(1+a(x))+\frac1{\min\{\varepsilon^{(\tilde{p}-p)/p},\varepsilon^{(\tilde{q}-q)/q}\}}\tilde{\mathcal H}(x,t)$$
and in turn for all $u\in L^{\tilde{\mathcal H}}(\Omega)$ 
\begin{equation}\label{rho}\varrho_\mathcal H(u)\le\varepsilon(|\Omega|+\|a\|_1)+ \frac1{\min\{\varepsilon^{(\tilde{p}-p)/p},\varepsilon^{(\tilde{q}-q)/q}\}}\varrho_{\tilde{\mathcal H}}(u).\end{equation}
First we put $\varepsilon:=1$ and we get for $u\neq0$, by \eqref{rho} and by the unit ball property, 
$$\varrho_\mathcal H\left(\frac{u}{\|u\|_{\tilde{\mathcal H}}}\right)\le|\Omega|+\|a\|_1+1,$$ 
and consequently, by (2.1.5) of \cite{Base},
$$\varrho_\mathcal H\left(\frac{u}{(|\Omega|+\|a\|_1+1)\|u\|_{\tilde{\mathcal H}}}\right)\le\frac1{|\Omega|+\|a\|_1+1}\varrho_\mathcal H\left(\frac{u}{\|u\|_{\tilde{\mathcal H}}}\right)\le1,$$
which again by the unit ball property gives
$$\|u\|_\mathcal H\le (|\Omega|+\|a\|_1+1)\|u\|_{\tilde{\mathcal H}}\quad\mbox{for all }u\in L^{\tilde{\mathcal H}}(\Omega).$$

\noindent For the second part of the statement, if $\tilde{q}/q\le\tilde{p}/p$, we fix $\varepsilon:= (\tilde{p}-p)/p \in (0,1)$ and we get 
$$\min\left\{\varepsilon^{\frac{\tilde{p}-p}p},\varepsilon^{\frac{\tilde{q}-q}q}\right\}= \left(\frac{\tilde{p}-p}p\right)^{\frac{\tilde{p}-p}p},$$
otherwise we fix $\varepsilon:= \frac{\tilde{q}-q}q \in (0,1)$ and we get 
$$\min\left\{\varepsilon^{\frac{\tilde{p}-p}p},\varepsilon^{\frac{\tilde{q}-q}q}\right\}= \left(\frac{\tilde{q}-q}q\right)^{\frac{\tilde{q}-q}q}.$$ 
Now, by the unit ball property and by \eqref{rho},
$$\varrho_\mathcal H\left(\frac{u}{\|u\|_{\tilde{\mathcal H}}}\right)\le C_{\tilde{\mathcal H},\mathcal H}\quad\mbox{for all } u\in L^{\tilde{\mathcal H}}(\Omega)\setminus\{0\}.$$
Therefore, being $\tilde{p}<2p$ and $\tilde{q}<2q$, $C_{\tilde{\mathcal H},\mathcal H}>1$ and so
$$\varrho_\mathcal H\left(\frac{u}{C_{\tilde{\mathcal H},\mathcal H}\|u\|_{\tilde{\mathcal H}}}\right)\le\frac1{C_{\tilde{\mathcal H},\mathcal H}}\varrho_\mathcal H\left(\frac{u}{\|u\|_{\tilde{\mathcal H}}}\right)\le 1.$$
The unit ball property then gives, 
$$\|u\|_\mathcal H\le C_{\tilde{\mathcal H},\mathcal H}\|u\|_{\tilde{\mathcal H}}\quad\mbox{for all }u\in L^{\tilde{\mathcal H}}(\Omega),$$
that is the thesis.
\end{proof}

\noindent
We now recall from \cite{dalmaso1993} the notion of $\Gamma$-convergence that will be useful in the sequel.

\begin{definition}\rm 
Let $X$ be a metrizable topological space and let $(f_h)$ be a sequence of functions from $X$ to $\overline{\mathbb{R}}$. The \emph{$\Gamma$-lower limit} and the \emph{$\Gamma$-upper limit} of the sequence $(f_h)$ are the functions from $X$ to $\overline{\mathbb{R}}$ defined by 
\begin{gather*}
\Big(\Gamma-\liminf_{h\to\infty} f_h\Big)(u)
= \sup_{U\in \mathcal{N}(u)}
\Big[\liminf_{h\to\infty}
\bigl(\inf\{f_h(v):v\in U\}
\bigr)\Big]\,,\\
\Big(\Gamma-\limsup_{h\to\infty} f_h\Big)(u)
=\sup_{U\in \mathcal{N}(u)}
\Big[\limsup_{h\to\infty}
\bigl(\inf\{f_h(v):v\in U\}\bigr)\Big]\,,
\end{gather*}
where $\mathcal{N}(u)$ denotes the family of all open neighborhoods of $u$ in $X$.
If there exists a function $f:X\to\overline{\R}$ such that 
$$
\Gamma-\liminf_{h\to\infty}f_h=\Gamma-\limsup_{h\to\infty}f_h=f,
$$
then we write $\Gamma-\lim\limits_{h\to\infty} f_h=f$
and we say that $(f_h)$ \emph{$\Gamma$-converges} to its \emph{$\Gamma$-limit} $f$.
\end{definition}

\begin{theorem}\label{gammaconv}
If $(p_h,q_h)\searrow (p,q)$, then
$$\mathscr E_\mathcal H(u)=\left(\Gamma-\lim_{h\to\infty}\mathscr E_{\mathcal H_h}\right)(u)\quad\mbox{for all }u\in L^1(\Omega).$$
\end{theorem}
\begin{proof} Let $u\in L^1(\Omega)$. First we prove that 
\begin{equation}\label{glimsup}
\mathscr E_\mathcal H(u)\ge\left(\Gamma-\limsup_{h\to\infty}\mathscr E_{\mathcal H_h}\right)(u).
\end{equation}
If $\mathscr E_\mathcal H(u)=\infty$, \eqref{glimsup} is immediate. Then, we suppose $\mathscr E_\mathcal H(u)<\infty$ and take $b\in\mathbb R$, $b>\mathscr E_\mathcal H(u)$. Let $\delta>0$ and $w\in C^1_0(\Omega)$ be such that $\|u-w\|_1<\delta$, with $\|\nabla w\|_\mathcal H<b$. By Lemma \ref{conver} we get $\|\nabla w\|_{\mathcal H_h}\to\|\nabla w\|_\mathcal H$ and so   
$$b>\mathscr E_\mathcal H(w)=\lim_{h\to\infty}\|\nabla w\|_{\mathcal H_h}=\lim_{h\to\infty}\mathscr E_{\mathcal H_h}(w)$$
and in turn
$$b>\limsup_{h\to\infty}(\inf\{\mathscr E_{\mathcal H_h}(v)\,:\,\|v-u\|_1<\delta\}).$$
By the arbitrariness of $b$ we conclude the proof of \eqref{glimsup}. Now, we want to prove that 
\begin{equation}\label{gliminf}
\mathscr E_\mathcal H(u)\le\left(\Gamma-\liminf_{h\to\infty}\mathscr E_{\mathcal H_h}\right)(u).
\end{equation} 
Suppose that $\left(\Gamma-\liminf_{h\to\infty}\mathscr E_{\mathcal H_h}\right)(u)<\infty$ (otherwise the conclusion is obvious) and take $b\in\mathbb R$, $b>\left(\Gamma-\liminf_{h\to\infty}\mathscr E_{\mathcal H_h}\right)(u)$. By Proposition 8.1-(b) of \cite{dalmaso1993} there is a sequence $(u_h)\subset L^1(\Omega)$ such that $u_h\to u$ in $L^1(\Omega)$ and 
$$\left(\Gamma-\liminf_{h\to\infty}\mathscr E_{\mathcal H_h}\right)(u)=\liminf_{h\to\infty}\mathscr E_{\mathcal H_h}(u_h).$$
Therefore, there exists a subsequence $((p_{h_j},q_{h_j}))$ such that $\mathscr E_{\mathcal H_{h_j}}(u_{h_j})<b$ for all $j$. 
Let $(v_j)\subset C^1_{0}(\Omega)$ be such that 
$$\|v_j-u_{h_j}\|_1<\frac1j,\quad\mathscr E_{\mathcal H_{h_j}}(v_j)<b\quad\mbox{for all }j\in\mathbb N.$$
Then, $v_j\to u$ in $L^1(\Omega)$ and by Lemma \ref{emb},  for $j$ sufficiently large,
\begin{equation}\label{bC}b>\|\nabla v_j\|_{\mathcal H_{h_j}}\ge\frac{\|\nabla v_j\|_\mathcal H}{C_{\mathcal H_{h_j},\mathcal H}}.\end{equation}
Furthermore, 
$$\|\nabla v_j\|_\mathcal H < b (|\Omega|+\|a\|_1+1),$$
thus $(v_j)$ is bounded in the reflexive Banach space $W^{1,\mathcal H}_0(\Omega)$, and so there exists a subsequence $(v_{j_m})$ such that $v_{j_m}\rightharpoonup u$ in $W^{1,\mathcal H}_0(\Omega)$. By \eqref{bC}, by the weak lower semicontinuity of the norm, and by the fact that $C_{\mathcal H_{h_{j_m}},\mathcal H}\to 1$,
$$b\ge\liminf_{m\to\infty}\frac{\|\nabla v_{j_m}\|_\mathcal H}{C_{\mathcal H_{h_{j_m}},\mathcal H}}\ge \|\nabla u\|_\mathcal H=\mathscr E_\mathcal H(u).$$
Thus, \eqref{gliminf} follows by the arbitrariness of $b$.
\end{proof}

\begin{lemma}\label{convergence}
Let $(p_h,q_h)\searrow(p,q)$, 
and $(u_h)\subset W^{1,\mathcal H}_0(\Omega)$ be a sequence such that  
$$\sup_{h\in\mathbb N}\|\nabla u_h\|_{\mathcal H}<\infty.$$
Then, there exist a subsequence $(u_{h_j})$ and a function $u\in W^{1,\mathcal H}_0(\Omega)$ such that
$$\lim_{j\to\infty}\varrho_{\mathcal H_{h_j}}(u_{h_j})=\varrho_\mathcal H(u).$$
\end{lemma}
\begin{proof} By hypotheses $(u_h)$ is bounded in the reflexive Banach space $W^{1,\mathcal H}_0(\Omega)$, thus there exists a subsequence $(u_{h_j})$ weakly convergent to $u\in W^{1,\mathcal H}_0(\Omega)$. By Proposition \ref{embeddingW}-$(iii)$, $W^{1,\mathcal H}_0(\Omega)\hookrightarrow\hookrightarrow L^{p^*-\varepsilon}(\Omega)$ for $\varepsilon>0$ sufficiently small. Hence, $u_{h_j}\to u$ in $L^{p^*-\varepsilon}(\Omega)$ and, up to a subsequence, 
$$\begin{gathered}u_{h_j}\to u\quad\mbox{a.e. in }\Omega,\\
|u_{h_j}|\le v\quad\mbox{for all }j\in\mathbb N,\mbox{ a.e. in }\Omega 
\end{gathered}$$
for some $v\in L^{p^*-\varepsilon}(\Omega)$.
Now, for $j$ sufficiently large, if we take $\varepsilon<(p^*-q)/2$, we get $$p_{h_j}<q_{h_j}\le q+\varepsilon<p^*-\varepsilon,$$ and consequently 
$$|u_{h_j}|^{p_{h_j}}+a(x)|u_{h_j}|^{q_{h_j}}\le 1+|v|^{p^*-\varepsilon}+\|a\|_\infty(1+|v|^{p^*-\varepsilon})\in L^1(\Omega).$$
Finally, by the dominated convergence theorem, we obtain
$$\lim_{j\to\infty}\int_\Omega\left(|u_{h_j}|^{p_{h_j}}+a(x)|u_{h_j}|^{q_{h_j}}\right)dx=\int_\Omega\left(|u|^p+a(x)|u|^q\right)dx.$$
\end{proof}

\begin{theorem}\label{teo2}
Let $(p_h,q_h)\searrow(p,q)$. Then, for every subsequence $((p_{h_j},q_{h_j}))$ and for every sequence $(u_j)\subset L^1(\Omega)$ verifying
$$\sup_{j\in\mathbb N}\mathscr E_{\mathcal H_{h_j}}(u_j)<\infty$$
there exists a subsequence $(u_{j_m})$ such that, as $m\to\infty$
$$\begin{gathered}
u_{j_m}\to u\quad\mbox{in }L^1(\Omega),\\
g_{\mathcal H_{h_{j_m}}}(u_{j_m})\to g_\mathcal H(u).
\end{gathered}$$
\end{theorem}
\begin{proof} Put $b:=\sup_{j\in\mathbb N}\mathscr E_{\mathcal H_{h_j}}(u_j)$ and for all $j$ we get by Lemma \ref{emb}
$$\|\nabla u_j\|_\mathcal H\le(|\Omega|+\|a\|_1+1)b<\infty,$$ 
that is $(u_j)$ is bounded in the reflexive Banach space $W^{1,\mathcal H}_0(\Omega)$. Thus, there exists a subsequence 
$(u_{j_m})$ such that $u_{j_m}\rightharpoonup u$ in $W^{1,\mathcal H}_0(\Omega)$, $u_{j_m}\to u$ in $L^1(\Omega)$ by Lemma \ref{embeddingW}-$(iii)$, and $u_{j_m}\to u$ a.e. in $\Omega$.

\noindent 
For the second part of the statement we have to prove that $\|u_{j_m}\|_{\mathcal H_{h_{j_m}}}\to\|u\|_\mathcal H$ up to a subsequence. 
In correspondence of $$\gamma>\liminf_{m}\|u_{j_m}\|_{\mathcal H_{h_{j_m}}},$$ we can find a subsequence, still denoted by $((p_{h_{j_m}},q_{h_{j_m}}))$, for which $\|u_{j_m}\|_{\mathcal H_{h_{j_m}}}<\gamma$ for all $m$. Therefore, $\varrho_{\mathcal H_{h_{j_m}}}(u_{j_m}/\gamma)<1$ and by Fatou's lemma
$$\int_\Omega\left(\left|\frac{u}\gamma\right|^p+a(x)\left|\dfrac{u}\gamma\right|^q\right)dx\le \liminf_{m\to\infty}\int_\Omega\left(\left|\dfrac{u_{j_m}}\gamma\right|^{p_{h_{j_m}}}+a(x)\left|\dfrac{u_{j_m}}\gamma\right|^{q_{h_{j_m}}}\right)dx\le1.$$
By the unit ball property $\|u\|_\mathcal H\le\gamma$ and by the arbitrariness of $\gamma$ we conclude that $$\|u\|_\mathcal H\le\liminf_m\|u_{j_m}\|_{\mathcal H_{h_{j_m}}}.$$ 
It remains to prove that $\|u\|_\mathcal H\ge\limsup_{m}\|u_{j_m}\|_{\mathcal H_{h_{j_m}}}$. Now, in correspondence of $\gamma<\limsup_{m}\|u_{j_m}\|_{\mathcal H_{h_{j_m}}}$, we can find a subsequence, still denoted by $((p_{h_{j_m}},q_{h_{j_m}}))$, for which $$\|u_{j_m}\|_{\mathcal H_{h_{j_m}}}>\gamma$$ and consequently $\varrho_{\mathcal H_{h_{j_m}}}(u_{j_m}/\gamma)>1$ for all $m$. By Lemma \ref{convergence}, up to a subsequence
$$\int_\Omega\left(\left|\frac{u}{\gamma}\right|^p+a(x)\left|\frac{u}{\gamma}\right|^q\right)dx=\lim_{m\to\infty}\int_\Omega\left(\left|\frac{u_{j_m}}{\gamma}\right|^{p_{h_{j_m}}}+a(x)\left|\frac{u_{j_m}}{\gamma}\right|^{q_{h_{j_m}}}\right)dx\ge1.$$ 
We conclude the proof by using the unit ball property and the arbitrariness of $\gamma$ as before. 
\end{proof}
\smallskip 

\noindent 
We are now ready to prove Theorem \ref{main2}.\smallskip

\noindent
$\bullet${\it Proof of Theorem \ref{main2}}. By Proposition \ref{g} and by the definition of $\mathscr E_\mathcal H$, the functionals $\mathscr E_\mathcal H$, $\mathscr E_{\mathcal H_h}$, $g_\mathcal H$ and $g_{\mathcal H_h}$ for all $h\in\mathbb N$ satisfy the structural assumptions required in Section 4 of \cite{Franzi}. Furthermore, Theorems~\ref{gammaconv} and \ref{teo2} prove that all hypotheses of Corollary 4.4 of \cite{Franzi} are verified and consequently 
$$\lim_{h\to\infty}\inf_{K\in\mathcal L^m_{\mathcal H_h}}\sup_{u\in K}\mathscr E_{\mathcal H_h}=\inf_{K\in\mathcal L^m_{\mathcal H}}\sup_{u\in K}\mathscr E_{\mathcal H},$$
where 
$$\mathcal L^m_\mathcal H:=\{K\subset L^1(\Omega)\cap\{\|u\|_\mathcal H=1\}\,:\,K\mbox{ compact, }K=-K,\, i(K)\ge m\}\quad\mbox{for all }m\in\mathbb N$$
(here $i$ is defined with respect to the $L^1(\Omega)$-topology), and the sets $\mathcal L^m_{\mathcal H_h}$ are defined analogously for all $h\in\mathbb N$.
Finally, by virtue of Proposition \ref{g} $(b)$, we can apply Corollary 3.3 of \cite{Franzi} to get that the minimax values with respect to the $L^1(\Omega)$-topology are the same as those with respect to the $W^{1,\mathcal H}_0(\Omega)$-topology and conclude the proof. \hfill$\Box$

\section{A Weyl-type law}\label{sec6}
\noindent
Throughout this section we assume that condition \eqref{cond-pq} is verified and that $\Omega$ is quasi-convex.
\begin{lemma}\label{le1}
For all $u\in W^{1,q}(\Omega)\setminus\{0\}$ we have
$$\frac1{w}\cdot\frac{\|\nabla u\|_p}{\|u\|_q}\le \frac{\|\nabla u\|_\mathcal H}{\|u\|_\mathcal H}\le w\frac{\|\nabla u\|_q}{\|u\|_p},$$
where $w:=1+\|a\|_\infty+|\Omega|$.
\end{lemma}
\begin{proof}
For all $u\in L^\mathcal H(\Omega)$
$$\int_\Omega|u|^pdx\le\int_\Omega(|u|^p+a(x)|u|^q)dx=\varrho_\mathcal H(u).$$
Hence, 
$\varrho_\mathcal H(u/\|u\|_p)\ge1$ for $u\neq0$,
and by the unit ball property
$\|u\|_\mathcal H\ge\|u\|_p$.
On the other hand, for all $u\in L^q(\Omega)$
$$\varrho_\mathcal H(u)=\int_\Omega(|u|^p+a(x)|u|^q)dx\le\int_\Omega[1+(1+\|a\|_\infty)|u|^q] dx=|\Omega|+(1+\|a\|_\infty)\|u\|_q^q$$
and so, if $u\neq 0$,
$$\varrho_\mathcal H\left(\frac{u}{\|u\|_q}\right)\le 1+\|a\|_\infty+|\Omega|,$$
whence, by the unit ball property and being $1+\|a\|_\infty+|\Omega|\ge1$,
$$\|u\|_\mathcal H\le(1+\|a\|_\infty+|\Omega|)\|u\|_q.$$
Hence, for all $u\in L^q(\Omega)$
$$\|u\|_p\le\|u\|_\mathcal H\le(1+\|a\|_\infty+|\Omega|)\|u\|_q,$$
which gives the thesis. 
\end{proof}

\begin{lemma}\label{le3}
Let $0<\delta<1$, consider the homothety $\Omega\to\delta\Omega$, $x\mapsto\delta x=:y$, and write $u(x)=v(y)$. Then, for all $u\in W^{1,q}(\Omega)\setminus\{0\}$
$$\frac{\|\nabla v\|_q}{\|v\|_p}=\delta^{-\sigma-1}\frac{\|\nabla u\|_q}{\|u\|_p},\qquad\frac{\|\nabla v\|_p}{\|v\|_q}=\delta^{\sigma-1}\frac{\|\nabla u\|_p}{\|u\|_q},$$
where $\sigma:=n\left(\dfrac1p-\dfrac1q\right)$.
\end{lemma}
\noindent The proof of the previous lemma follows by straightforward calculations. Furthermore, we notice that $\sigma\in(0,1)$, being $p<q<p^*$. 

\smallskip

\noindent
We introduce now the auxiliary problem in $W^{1,\mathcal H}(\Omega)$
$$\begin{aligned}-\mathrm{div}&\left(\left[p\left(\frac{|\nabla u|}{L(u)}\right)^{p-2}+qa(x)\left(\frac{|\nabla u|}{L(u)}\right)^{q-2}\right]\frac{\nabla u}{L(u)}\right)\\
&\hspace{3cm}=\lambda T(u) \left[p\left(\frac{|u|}{\ell(u)}\right)^{p-2}+qa(x)\left(\frac{|u|}{\ell(u)}\right)^{q-2}\right]\frac u{\ell(u)},\end{aligned}$$
where 
$L(u):=\|\nabla u\|_\mathcal H$, $\ell(u):=\|u\|_\mathcal H$, and $$T(u):=\frac{\displaystyle\int_\Omega\left[p\left(\frac{|\nabla u|}{L(u)}\right)^p+qa(x)\left(\frac{|\nabla u|}{L(u)}\right)^q\right]dx}{\displaystyle\int_\Omega\left[p\left(\frac{|u|}{\ell(u)}\right)^p+qa(x)\left(\frac{|u|}{\ell(u)}\right)^q\right]dx}.$$
The eigenvalues and eigenfunctions of this problem on 
$$\mathcal N:=\{u\in W^{1,\mathcal H}(\Omega)\,:\,\ell(u)=1\}$$
are critical values and critical points of $\widetilde{L}:=L|_\mathcal N$. 
Furthermore, we set for every $\lambda,\,\mu\in\mathbb R$
$$
\begin{aligned}&\widetilde{K}^\lambda:=\{u\in\mathcal M\,:\,\widetilde{K}(u)<\lambda\},\quad\widetilde{L}^\mu:=\{u\in\mathcal N\,:\,\widetilde{L}(u)<\mu\},\\
&\widehat{K}(u):=\|\nabla u\|_q\quad\mbox{for }u\in\widehat{\mathcal M}:=\{W^{1,q}_0(\Omega)\,:\,\|u\|_p=1\},\\
&\widehat{L}(u):=\|\nabla u\|_p\quad\mbox{for }u\in\widehat{\mathcal N}:=\{u\in W^{1,q}(\Omega)\,:\,\|u\|_q=1\},\\
&\widehat{K}^\lambda:=\{u\in\widehat{\mathcal M}\,:\,\widehat{K}(u)<\lambda\}\quad\mbox{and}\quad\widehat{L}^\mu:=\{u\in\widehat{\mathcal N}\,:\,\widehat{L}(u)<\mu\}.
\end{aligned}$$
For all $\lambda\in(\lambda^m_\mathcal H,\lambda^{m+1}_\mathcal H]$,  $i(\widetilde{K}^{\lambda})=m$, (see Proposition 3.53 of \cite{PAORbook} and \cite{MR1017063}). Therefore,  
\begin{equation}\label{count}i(\widetilde{K}^\lambda)=\sharp\{m:\lambda^m_\mathcal H<\lambda\}\quad\mbox{for all }\lambda\in\mathbb R.\end{equation}

\noindent 
We recall here the definitions of two topological invariants of symmetric sets which will be useful in the next proofs. 

\begin{definition}
{\rm For every nonempty and symmetric subset $A$ of a Banach space $X$, its {\it cogenus} is defined by
\begin{equation}
\label{cogenus}
\overline{\gamma}(A)=\inf\left\{k\in\mathbb{N}\, :\, \exists \mbox{ a continuous odd map } f:\mathbb{S}^{k-1}\to A \right\},
\end{equation}
with the convention that $\overline{\gamma}(A):=+\infty$, if no such an integer $k$ exists. 
}
\end{definition}

\begin{definition}[ cf. \cite{fadell_rabinowitz1978}]
{\rm For every closed symmetric subset $A$ of a Banach space $X$, we define the quotient space $\overline{A}:=A/\mathbb Z_2$ with each $u$ and $-u$ identified. Let $F:\overline{A}\to\mathbb R P^\infty$ be the classifying map of $\overline{A}$ towards the infinite-dimensional
projective space, which induces a homomorphism of the Alexander-Spanier cohomology rings $f^* :H^*(\mathbb R P^\infty) \to H^*(\overline{A})$. One
can identify $H^*(\mathbb R P^\infty)$ with the polynomial ring $Z_2[\omega]$ on a single generator $\omega$. Finally, we define the
cohomological index of $A$
$$g(A):=\begin{cases}\sup\{k\in\mathbb N\,:\,f^*(\omega^{m-1})\neq 0\},\quad&\mbox{if }A\neq\emptyset,\\
0,\quad&\mbox{if }A=\emptyset.\end{cases}$$}
\end{definition}

\noindent 
In what follows the topological index $i$ used in definition \eqref{mth-eigen} will be denoted by $\gamma$, when it stands for the Krasnosel'ski\u{\i} genus, and by $g$, when it stands for the $\mathbb Z_2$-cohomological index of Fadell and Rabinowitz. \smallskip

\noindent
By the monotonicity -- property ($i_2$) -- and the supervariance -- property ($i_3$) -- of the cohomological index, and by the fact that $g(\mathbb S^{k-1})=k$, for any symmetric subset $A$ of a Banach space $X$, with finite genus, cogenus and cohomological index, it results 
\begin{equation}\label{gammag}\overline{\gamma}(A)\le g(A)\le\gamma(A).\end{equation} 

\begin{proposition}\label{density} $C^\infty(\overline{\Omega})$ is dense in $W^{1,\mathcal H}(\Omega)$.
\end{proposition}
\begin{proof} By Proposition 4.1 of \cite{HHK} there exists a suitable extension $\tilde{\mathcal H}$ of $\mathcal H$ to all of $\mathbb R^n\times[0,\infty)$. For $u\in W^{1,\mathcal H}(\Omega)$, let $\tilde u\in W^{1,\tilde{\mathcal H}}(\mathbb R^n)$ denote an extension of $u$ such that
$$
\|\tilde u\|_{W^{1,\tilde{\mathcal H}}(\mathbb R^n)}\le c\|u\|_{W^{1,\mathcal H}(\Omega)}.
$$
By Theorem~5.5 of \cite{HHK}, $C^\infty_0(\mathbb R^n)$ is dense in $W^{1,\tilde{\mathcal H}}(\mathbb R^n)$, so we can find $(\tilde{u}_h)\subset C^\infty_0(\mathbb R^n)$ such that $\tilde{u}_h\to\tilde u$ in $W^{1,\tilde{\mathcal H}}(\mathbb R^n)$. Therefore, put $u_h:=\tilde{u}_h\big|_\Omega$ for all $h$ and have
$$\|u-u_h\|_{W^{1,\mathcal H}(\Omega)}\le\|\tilde u-\tilde u_h\|_{W^{1,\tilde{\mathcal H}}(\mathbb R^n)}\to 0 \quad\mbox{ as }h\to\infty.$$
This proves that $u_h$ are the required approximating functions. 
\end{proof}
\noindent 
For an alternative proof of the previous result see also Theorem 2.6 of \cite{thai}.

\begin{lemma}\label{le2} For all $\lambda\in\mathbb R$
$$\overline{\gamma}(\widehat{K}^{\lambda/w})\le\overline{\gamma}(\widetilde{K}^\lambda),\qquad\gamma(\widetilde{L}^\lambda)\le\gamma(\widehat{L}^{w\lambda}).$$
\end{lemma}
\begin{proof} By Lemma \ref{le1}, the maps 
$$\widehat{K}^{\lambda/w}\to\widetilde{K}^\lambda,\quad u\mapsto\frac{u}{\|u\|_\mathcal H} \quad\mbox{ and }\quad \widetilde{L}^\lambda\cap W^{1,q}(\Omega)\to \widehat{L}^{w\lambda},\quad u\mapsto\frac{u}{\|u\|_q}$$
are well defined, odd, and continuous. Moreover, by Proposition \ref{density} and by the fact that $W^{1,q}(\Omega)\hookrightarrow W^{1,\mathcal H}_0(\Omega)$ (cf. Proposition \ref{embeddingW}-$(v)$), $W^{1,q}(\Omega)$ is dense in $W^{1,\mathcal H}(\Omega)$ and so the inclusion $\widetilde{L}^\lambda\cap W^{1,q}(\Omega)\subset\widetilde{L}^\lambda$ is a homotopy equivalence by virtue of Theorem 17 of \cite{palais}. The conclusion follows 
by the definition of $\gamma$ and $\overline{\gamma}$.
\end{proof}

\begin{lemma}\label{le4} 
If $\Omega_1$ and $\Omega_2$ are disjoint subdomains of $\Omega$ such that $\overline{\Omega}_1\cup\overline{\Omega}_2=\overline{\Omega}$, then 
$$\overline{\gamma}(\widehat{K}_{\Omega_1}^\lambda)+\overline{\gamma}(\widehat{K}_{\Omega_2}^\lambda)\le\overline{\gamma}(\widehat{K}^\lambda), \quad
\gamma(\widehat{L}^\lambda)\le\gamma(\widehat{L}_{\Omega_1}^{\lambda'})+\gamma(\widehat{L}_{\Omega_2}^{\lambda'})\quad \mbox{ for all }\lambda<\lambda',$$
where the subscripts indicate the corresponding domains and we drop the subscript when the domain is $\Omega$.
\end{lemma}
\noindent The proof of the previous lemma can be obtained reasoning as in the proof of Lemma 3.2 of \cite{persqu}, by simply replacing $p^+$ with $q$, $p^-$ with $p$ and $p(x)$ with $\mathcal H(x,t)$, hence we omit it.

\medskip
\noindent
We can now prove the last result of the paper. The proof relies on the argument produced for Theorem 1.1 of \cite{persqu}, we report it here for the sake of completeness.
\smallskip

\noindent  
$\bullet${\it Proof of Theorem \ref{weyl}}. 
The proof is split in two steps, first we prove the statement for $n$-dimensional cubes and then we approximate the domain $\Omega$ by unions of cubes. 

\noindent\underline{\it Step 1.} Let $Q$ be the unit cube in $\mathbb R^n$, fix $\lambda_0>\max\{\inf \widehat{K}_Q,\inf\widehat{L}_Q\}$, and set 
$$r:=\overline{\gamma}(\widehat{K}^{\lambda_0}_Q),\quad s:=\gamma(\widehat{L}_Q^{\lambda_0}).$$
Then, take $\lambda\in(\lambda_0,\lambda')$ and any two cubes $Q_{a_\lambda}$ and $Q_{b_{\lambda'}}$ of sides $a_\lambda:=(\lambda_0/\lambda)^{1/(1+\sigma)}<1$ and $b_{\lambda'}:=(\lambda_0/\lambda')^{1/(1-\sigma)}<1$, respectively. By lemma \ref{le3} it is easy to check that the functions
$$\widehat{K}_Q^{\lambda_0}\to\widehat{K}^\lambda_{Q_{a_\lambda}},\;\; u\mapsto\frac{v}{\|v\|_p}\quad\mbox{ and }\quad\widehat{L}^{\lambda_0}_Q\to\widehat{L}^{\lambda'}_{Q_{b_{\lambda'}}},\;\;u\mapsto\frac{v}{\|v\|_q}$$
are odd homeomorphisms, and so, by property ($i_3$) of the topological index, we obtain
$$\overline{\gamma}(\widehat{K}_{Q_{a_\lambda}}^\lambda)=r,\quad\gamma(\widehat{L}^{\lambda'}_{Q_{b_{\lambda'}}})=s.$$
Therefore, by Lemma \ref{le4}, if we denote by $Q_a$ a cube of side $a>0$, 
$$r\left[\frac{a}{a_\lambda}\right]^n \le\overline{\gamma}(\widehat{K}^\lambda_{Q_a}),\qquad \gamma(\widehat{L}^\lambda_{Q_a})\le s\left(\left[\frac{a}{b_{\lambda'}}\right]+1\right)^n,$$
Whence, for $\lambda'$ and $\lambda$ large 
\begin{equation}\label{3.4}C_1a^n\lambda^{n/(1+\sigma)}\le\overline{\gamma}(\widehat{K}^\lambda_{Q_a}),\qquad\gamma(\widehat{L}^\lambda_{Q_a})\le C_2 a^n(\lambda')^{n/(1-\sigma)},\end{equation}
with $C_1:=r/\lambda_0^{n/(1+\sigma)}$ and $C_2:=s/\lambda_0^{n/(1-\sigma)}$ depending only on $n$, $p$ and $q$.

\noindent\underline{\it Step 2.} Let $\varepsilon>0$ and let $\Omega_\varepsilon, \,\Omega^\varepsilon$ be finite unions of cubes with pairwise disjoint interiors such that $$\Omega_\varepsilon:=\bigcup_{j=1}^{M_\varepsilon}Q_j\subset\Omega\subset\Omega^\varepsilon:=\bigcup_{j=1}^{M^\varepsilon}Q'_j$$ and $|\Omega^\varepsilon\setminus\Omega_\varepsilon|<\varepsilon$. Then, by \eqref{3.4}, Lemma \ref{le4} and the monotonicity of $\overline{\gamma}$ 
$$\begin{gathered}C_1|\Omega_\varepsilon|\lambda^{n/(1+\sigma)}\le\sum_{j=1}^{M_\varepsilon}\overline{\gamma}(\widehat{K}^{\lambda}_{Q_j})\le\overline{\gamma}(\widehat{K}_{\Omega_\varepsilon}^\lambda)\le\overline{\gamma}(\widehat{K}^\lambda),\\
\gamma(\widehat{L}^\lambda)\le\gamma(\widehat{L}^\lambda_{\Omega^\varepsilon})\le\sum_{j=1}^{M^\varepsilon}\gamma(\widehat{L}^{\lambda'}_{Q'_j})\le C_2|\Omega^\varepsilon|(\lambda')^{n/(1-\sigma)}.\end{gathered}$$
By Tietze theorem, we can extend continuously $a(\cdot)$ to all of $\mathbb R^n$ and obtain a nonnegative function having the same $L^\infty$-norm as $a(\cdot)$.
Thus, by the arbitrariness of $\varepsilon>0$ and of $\lambda'>\lambda$, we get
$$\begin{aligned}C_1|\Omega|&\lambda^{n/(1+\sigma)}\le\overline{\gamma}(\widehat{K}^\lambda)\le\overline{\gamma}(\widetilde{K}^{w\lambda})\le g(\widetilde{K}^{w\lambda})\\
&\le g(\widetilde{L}^{w\lambda})\le\gamma(\widetilde{L}^{w\lambda})\le\gamma(\widehat{L}^{w^2\lambda})\le C_2|\Omega|(w^2\lambda)^{n/(1-\sigma)},\end{aligned}$$
where we have used Lemma \ref{le2}, inequalities \eqref{gammag}, the fact that $\widetilde{K}^\lambda\subset\widetilde{L}^\lambda$, and the monotonicity of $g$. Finally, the conclusion follows by \eqref{count}.\hfill$\Box$

\bigskip
\bigskip

\end{document}